\documentclass[english]{article}
\usepackage[utf8]{inputenc}
\usepackage[T1]{fontenc}
\usepackage{lmodern}
\usepackage[a4paper]{geometry}
\usepackage[english]{babel}
\usepackage{csquotes}
\usepackage{graphicx}
\usepackage[verbose]{wrapfig}

\usepackage{amsmath}
\usepackage{amssymb}
\usepackage{amsthm}
\usepackage{stmaryrd}
\usepackage{empheq}
\usepackage{microtype}

\usepackage{soul}
\newcommand{\com}[1]{}

\usepackage{tikz-cd}

\usepackage{hyperref}
\usepackage{orcidlink}
\usepackage[affil-it]{authblk}

\newcommand{\setR}{\mathbb{R}}
\newcommand{\setZ}{\mathbb{Z}}
\newcommand{\setN}{\mathbb{N}}

\newcommand{\setH}{\mathbb{H}}
\newcommand{\setS}{\mathbb{S}}

\newcommand{\h}{\xi}
\newcommand{\bh}{{\bar{\h}}}
\newcommand{\bz}{{\bar{\zeta}}}
\newcommand{\Zdeux}{{\setZ/2\setZ}}

\newcommand{\isom}{{\xrightarrow{\sim}}}

\newcommand{\met}{{\eta}}

\newcommand{\Lie}{\mathcal L}
\DeclareMathOperator{\ExT}{\Lambda}

\renewcommand{\d}{{\operatorname d}} 

\newcommand{\dom}{\d^\omega}

\newcommand{\Glie}{\mathfrak g}
\newcommand{\Hlie}{\mathfrak h}
\newcommand{\Klie}{\mathfrak k}

\newcommand{\lp}{\left(}
\newcommand{\rp}{\right)}

\newcommand{\wb}[2]{\left[#1\wedge#2 \right]}
\newcommand{\fwb}[2]{\frac{1}{2}\wb{#1}{#2}}

\newcommand{\U}{\mathcal U}

\DeclareMathOperator{\Riem}{Riem}
\DeclareMathOperator{\Ric}{Ric}
\DeclareMathOperator{\Scal}{Scal}
\DeclareMathOperator{\Ein}{Ein}

\DeclareMathOperator{\ad}{ad}
\DeclareMathOperator{\Ad}{Ad}
\DeclareMathOperator{\End}{End}

\DeclareMathOperator{\Mat}{M}
\DeclareMathOperator{\Spin}{Spin}

\DeclareMathOperator{\GL}{GL}

\DeclareMathOperator{\SpH}{Sp}

\DeclareMathOperator{\Stab}{Stab}
\DeclareMathOperator{\id}{id}

\renewcommand{\varphi}{\phi}

\DeclareMathOperator{\SO}{SO}

\renewcommand{\so}{\operatorname{\mathfrak{so}}}
\newcommand{\spin}{\operatorname{\mathfrak{spin}}}

\DeclareMathOperator{\Mink}{\m}
\DeclareMathOperator{\Cl}{Cl}

\newcommand{\m}{\Mink}

\newcommand{\Pprin}{P_{prin}}

\newcommand{\Orb}{\mathcal{O}}

\newcommand{\PMQ}{P\times_M Q}
\newcommand{\Omhor}{\Omega_\text{hor}}
\newcommand{\Ombas}{\Omega_\text{bas}}

\newtheorem{theorem}{Theorem}[section]
\newtheorem{proposition}[theorem]{Proposition}
\theoremstyle{definition}
\newtheorem{definition}[theorem]{Definition}
\newtheorem{example}[theorem]{Example}

\newtheorem{lemma}[theorem]{Lemma}

\theoremstyle{remark}
\newtheorem*{remark}{Remark}

\usepackage[style=alphabetic]{biblatex}
\title{A local generalisation of frame bundles}
\author{Jérémie Pierard de Maujouy, \orcidlink{0000-0002-2996-0262}}
\affil{Université de Tours - Institut Denis Poisson\\
Facultés des Sciences et Techniques\\
Parc de Grandmont\\
37200 Tours, France}

\date{\today}

\begin{document}

\maketitle

\begin{abstract}
	Frame bundles equipped with a principal connection have their local structure characterised by a $1$-form, called the Cartan connection $1$-form, which gathers the principal connection form and the soldering form. We introduce \emph{generalised frame bundles} as a smooth manifold equipped with a coframe with value in a suitable Lie algebra and which furthermore satisfies a weakened version of the Maurer-Cartan equation. From this structure it is possible to construct a Lie algebra action on the manifold. We study the question of whether it is possible to construct a Lie group action and build a base manifold over which the initial manifold is a frame bundle. We find that generalised frame bundles can have singular underlying manifolds. 
\end{abstract}

\tableofcontents

\com{Notation pour les champs de vecteurs : $\Gamma(TM)$ ?}

\section{Introduction}\label{secno:Intro}
\subsection{Background and motivation}

Principal bundles have a preponderant position in differential geometry.
They are a geometric realisation of an ambiguity often called \enquote{gauge} and parametrised by a Lie group.
The standard model of particle physics is famously a \emph{gauge theory}, which relies crucially on the gauge principle.
Principal bundles are also omnipresent in pure differential geometry: the Hopf fibration is an example of a $\setS^1$-principal fibre bundle, another example is the bundle of linear frames of a manifold, which embodies the non-unicity of local frames.

At least as important as principal bundles themselves is the notion of principal connection and its associated covariant derivations. 
It is no overstatement to say that connections are the central object of gauge theories and that principal bundles are essentially viewed as support for them.
In Riemannian geometry as well, the Levi-Civita connection is sees way more use than the principal bundle that supports it, the bundle of orthonormal frames. It therefore makes sense to consider principal bundles equipped with principal connections as a single object.

Cartan geometry describes manifolds as a \enquote{curved} version of homogeneous spaces, much like Riemannian geometry can be understood as a curved version of Euclidean geometry. It uses principal bundles, which describe the different ways the manifold can be related to the model geometry, that are equipped with a different type of connection: a Cartan connection. The Cartan connection \enquote{solders} the manifold to the model geometry, by encoding how the model geometry can slide along a contact point with the manifold. 

In this article, we look at a generalisation of Cartan geometry, in the specific case of frame bundles with affine connections.
The generalisation is very simple: we just formalise the local structure of frame bundles with affine connections.
A space with such a structure will be called \enquote{generalised frame bundle} (with connection -- usually not-mentioned).
Since the structure of principal bundle is non-local, it cannot be captured by a purely local definition, and generalised frame bundle are not necessarily principal bundle: this is one of the matters we investigate in this article.

Why this generalisation?
It turns out that generalised frame bundles are characterised by a vector-valued $1$-form which satisfies a system of partial differential equations, along with a non-degeneracy condition. This allows for a very concise definition. Furthermore, this system of partial differential equations can be obtained as Euler-Lagrange equations. This is in fact the starting point for defining generalised frame bundles: they are naturally produced in a Lagrangian field theory which looks to produce the space-time of General Relativity by constructing its bundle of orthonormal frames, starting from a blank $10$-manifold~\cite{LFB, PDMPhD, PDMDiracSpinors}.

A benefit of this generalisation is that by setting the focus on the total space of the principal bundle, it allows for base spaces which are more singular than smooth manifolds, such as orbifolds as we will discuss.

This article aims to investigate the following questions: when is a generalised frame bundle a standard frame bundle? If not, can it be related to a standard frame bundle? What would be the underlying base space?
These questions will involve the problem of integrating a Lie algebra action into a Lie group action. This problem was studied in detail by Palais~\cite{GlobalLie} and was revisited by Kamber and Michor~\cite{LieActionInt}. The author will re-examine this matter in an upcoming paper; in the present article, we will only briefly describe the general results.

The structure of generalised frame bundles turns out to be a case of Alekseevsky and Michor's definition of Cartan connections~\cite{DiffGeoCartan}. However, while they study characteristic classes and prolongations, we are interested in comparing the global geometry with the familiar structure of frame bundles. 

\paragraph{Outline}
In Section~\ref{secno:Intro}, after general definitions, we will give an elementary illustration of a generalised frame bundle, with pictures.
In Section~\ref{secno:FB&GFB}, after a reminder on frame bundles and $G$-structures, we introduce the generalised frames bundles and standard constructions on them.
We describe a few examples in Section~\ref{secno:examples}. The remaining sections are dedicated to the question of when a given generalised frame bundle is a standard frame bundle.
In Section~\ref{secno:HManifolds} we briefly discuss the main results on the integration of a Lie algebra action to a Lie group action then give the general results on the quotient of a manifold by a proper group action. These results are applied in Section~\ref{secno:CartanInteg} to the case of a generalised frame bundle, and we prove that under certain conditions, there exists a dense open subset of a generalised frame bundle which forms a standard frame bundle above its orbit manifold. We finally discuss the case of a compact group and a compact manifold, in which many requirements are automatically satisfied.

\paragraph{Acknowledgements}
This paper presents results obtained by the author during his research as a PhD student at Université Paris Cité. Many of the ideas originate from discussions with Frédéric Hélein, the PhD supervisor.

\subsection{Conventions and definitions}


When not specified, we shall assume smooth manifolds to be \emph{Hausdorff paracompact}.
%
%
%
%


In this paper \emph{Lie algebra} will mean \ul{finite dimensional real Lie algebra}. Let us recall that for a Lie group $G$, its Lie algebra
$\Glie$ can be identified as the tangent space at identity $T_e G$ provided with the infinitesimal adjoint action bracket: 
\[
	[\h_1,\h_2] = \ad_{\h_1}(\h_2)
\]
It is naturally isomorphic to the Lie algebra of \emph{left}-invariant vector fields on $G$ (hence corresponding to a \emph{right} action on $G$). On the other hand \emph{right}-invariant vector fields on $G$ form a Lie algebra which is anti-isomorphic to $\Glie$ in the following sense: writing $L_\h$ for the right-invariant vector field such that $L_\h (e) = \h$, the following holds:
\[
	[ L_{\h_1} , L_{\h_2} ] = L_{ - [\h_1,\h_2] }
\]

Since we are interested in principal bundles and similar structures, groups and algebras will by default act \emph{on the left} on vector spaces and \emph{on the right} on manifolds.

\begin{definition}[$G$-manifold]
Given a Lie group $G$, a left/right \emph{$G$-manifold} is a smooth manifold
$M$ with a smooth left/right action of $G$, namely a smooth map 
%
\begin{gather*}
	\rho : G\times M \to M \text{  such that  } 
	\rho(g_1g_2,x)=\rho(g_1,\rho(g_2,x))
	\qquad \text{(left action)}
\\
	\rho : M\times G \to M \text{  such that  } 
	\rho(x,g_1g_2)=\rho(\rho(x,g_1), g_2)
	\qquad \text{(right action)}
\end{gather*}
We will often use an implicit notation $x \cdot g$ (or $g \cdot x$) for the action of $g$ on $x$.
\end{definition}
All Lie group actions we will consider are smooth.

\begin{definition}[$\Glie$-manifold]
Denoting by $\Glie$ a Lie algebra, a \emph{right} $\Glie$-manifold is a (smooth) manifold with an action of $\Glie$ by smooth vector fields:
\begin{gather*}
	\h\in \Glie \to \bh \in \Gamma(TX) \text{  such that  } 
	[\bh_1,\bh_2] = \overline{[\h_1,\h_2]}
\end{gather*}

The vector fields $\bh$ are called \emph{fundamental vector fields} on the $\Glie$-manifold.
\end{definition}

\begin{remark}
	A \emph{right} smooth action of a Lie group differentiates to a right action of the associated Lie algebra.
	
	Similarly, a smooth group action on the left differentiates to an action of the associated Lie algebra which reverses the bracket.
\end{remark}

Let $M$ be a $G$-manifold. To each point $x\in M$ is associated an \emph{isotropy group} 
\[
	G_x := \{ g\in G| x\cdot g = x \}
\]
and an \emph{orbital map}
\[
	g\in G\mapsto x\cdot g\in M
\]
which naturally factors through a smooth injection $G_x \backslash G \to M$ onto the orbit of $x$.

\begin{definition}[Properties of actions of Lie groups and algebras]
The action $\rho$ of a group $G$ (resp. a Lie algebra $\Glie$) on a manifold $M$ is said to be:
\begin{itemize}
	\item \emph{effective} (also \emph{faithful}) if no non-trivial element acts trivially:
		\begin{align*}
			\forall g\in G, \quad \rho(g) &= \id_M \implies g = e\\
			\forall \h\in \Glie, \quad \rho(\h) &= 0_{TM} \implies \h = 0
		\end{align*}
		
	\item \emph{free} if the isotropy group (resp. algebra) at each point is trivial:
		\begin{align*}
			\forall x\in M, \forall g\in G,& \quad x\cdot g = x \implies g = e\\
			\forall x\in M, \forall \h\in \Glie,& \quad
			\bh|_x =0 \implies \h = 0
		\end{align*}
		
	\item \emph{transitive} if the orbital maps are surjective (resp. the vectors representing the Lie algebra span the whole tangent space at each point \footnote{This is sometimes called \emph{locally transitive} or \emph{infinitesimally transitive}.}):
		\begin{align*}
			\forall x,y \in M, \ \exists g\in G,\quad x\cdot g &= y\\
			\forall x\in M, \forall X\in T_x M, \ \exists \h \in \Glie, \quad x \cdot \h &= X
		\end{align*}
\end{itemize}
\end{definition}

Furthermore, the action of a Lie algebra on a manifold is said to be \emph{complete} when the flow of every fundamental vectors fields is complete. We also say that the $\Glie$-manifold is complete.

\begin{definition}[Principal $G$-bundle]
Given a Lie group $G$, a principal $G$-bundle is a fibre bundle $P\xrightarrow{p} M$ with $G$ acting \emph{freely} on $P$ such that the fibres of $p$ are the orbits under $G$. We will follow the usual convention that principal $G$-bundles have a \emph{right} group action.
\end{definition}

\begin{definition}[Equivariant bundle]
An \emph{equivariant bundle} on a $G$-manifold $M$ is a fibre bundle $P\xrightarrow{p} M$ endowed with an action of $G$ which lifts the action on $M$:
\[ \forall (g,x)\in G\times P, \quad p(x\cdot g) = p(x)\cdot g  \]
Namely, the following diagram commutes:
\begin{center}\begin{tikzcd}
	P\times G 
		\ar[r, "\rho_P"]
		\ar[d,"p \times \id_G"] &
	G
		\ar[d,"p"]
	\\
	M \times G
		\ar[r, "\rho"] &
	M		
\end{tikzcd}\end{center}
An \emph{equivariant section} of an equivariant bundle is a section which is invariant under the action of $G$: 
\[ \phi : M\to P\, \text{ such that } \forall x\in M, \;
	\begin{cases}
		p(\phi(x)) = x\\
		\phi(x\cdot g) = \phi(x) \cdot g
	\end{cases}
\]
Namely, the following diagram commutes: 
\begin{center}\begin{tikzcd}
	P\times G 
		\ar[r] &
	P
	\\
	M \times G
		\ar[r] 
		\ar[u,"\phi \times \id_G"] &
	M
		\ar[u,"\phi"]
\end{tikzcd}\end{center}

\end{definition}
We will be mainly be interested in equivariant sections of bundles of the form $\ExT^\bullet T^*M\otimes V$ with $V$ a linear representation of $G$.
\subsection{An illustrative example of a generalised frame bundle}\label{secno:toymodel}

This article studies the structure of \enquote{generalised frame bundle} which is a generalisation of principal bundles equipped with a Cartan connection. We depict in this first section a simple and very amenable case of this kind of structure constructed by performing a \enquote{twist} on a frame bundle.

\subsubsection*{The orthonormal direct frame bundle of $\setR^2$}

\begin{wrapfigure}{R}{0.4\textwidth}
	\centering
	\caption{The frame bundle of $\setR^2$ as an (unfolded) bulk torus}\label{figno:torus1}
	\includegraphics[width=0.4\textwidth]{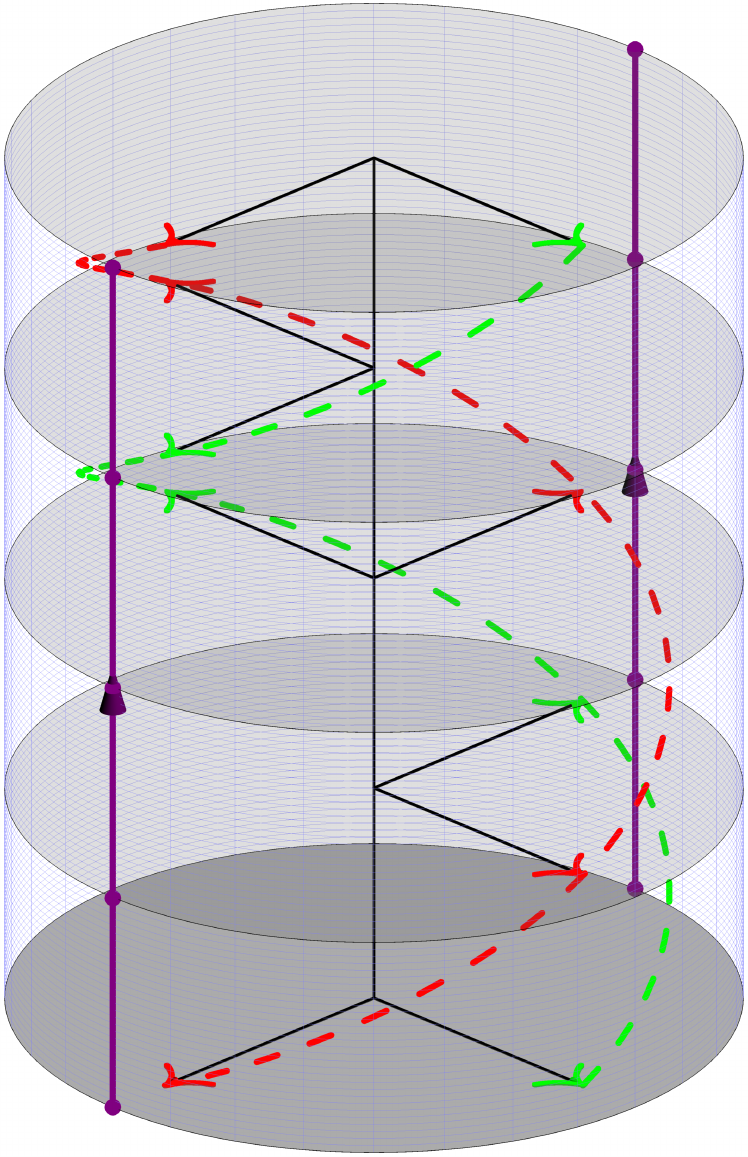}
\end{wrapfigure}

We will consider the plane as a smooth manifold $M=\setR^2$ which we will represent as an open disk. It is equipped with its standard oriented Euclidean structure. This allows us to define its (orthonormal direct) frame bundle $P$, which is the space of direct orthonormal bases of the tangent space at each point. The various frames at a given point are related by rotations, and the frame bundle forms a smooth $\SO(2)$-principal bundle over $M$: 
\[
	\SO(2) \hookrightarrow P \overset{\pi}\twoheadrightarrow M
\]

This bundle turns out to be trivialisable in many ways (for example because $M$ is contractile) and can thus be represented as an open solid torus $M\times \SO(2)$, generated by the revolution of $M$. Since $P$ is a space of frames, each point $p$ defines a frame of $T_{\pi(p)} M$. Different frames above $\pi(p)$ are related by the action of $\SO(2)$. This is represented as a cylinder in Figure~\ref{figno:torus1}: the top and bottom face should be identified, and revolutions take the form of vertical translations.
In particular, the trivialised frame bundle comes with a parallelism, namely smooth vector fields which constitute at each point a linear frame \emph{of $P$}. At each point, its first two vectors are the two frame vectors along the section (in red and green in Figure~\ref{figno:torus1}), we call them $e_1$ and $e_2$. The third vector is the normalised vector transverse to the section, generator of the revolution action of $\SO(2)$ (in purple on Figure~\ref{figno:torus1}). We will call it $\xi$.

The structure of the frame bundle imposes that the frame $(e_1,e_2)$ be \emph{equivariant} under rotation around the revolution axis. If we call $R(\theta)$ the revolution of angle $\theta\in \setR$ this means that:

\begin{subequations}
\begin{empheq}[left=\empheqbiglbrace]{align}
	R(\theta)^* e_1 &= \cos(\theta)e_1 + \sin(\theta)e_2\\ 
	R(\theta)^* e_2 &= -\sin(\theta)e_1 + \cos(\theta)e_2
\end{empheq}
\end{subequations}

Effectively, the frame of $M$ is rotating when progressing along the revolution. Since $\xi$ generates the action of $\SO(2)$, there is an infinitesimal version of the equivariance equations:
\begin{subequations}\label{eqno:Zequiv}
\begin{empheq}[left=\empheqbiglbrace]{align}
	[\xi,e_1] &= e_2\\ 
	[\xi,e_2] &= -e_1
\end{empheq}
\end{subequations}
The two sets of equations are equivalent since $R(\theta)$ is the flow $\exp(\theta \xi)$.


Now, we want to present a different representation of this situation. The frame rotation's can be \enquote{untwisted} by applying a diffeomorphism of the solid torus (with nontrivial mapping class). Using the standard coordinates $(x,y)$ on the section $M$ and a cyclic coordinate $z$ on $\SO(2)$, the diffeomorphism takes the following form:
\begin{empheq}[left=\empheqbiglbrace]{align*}
	M\times\SO(2) &\to M\times \SO(2)\\
	\begin{pmatrix} x\\y\\z \end{pmatrix}
	&\mapsto
	\begin{pmatrix}
	\cos(z)x+\sin(z)y \\ -\sin(z)x+\cos(z)y \\ z
	\end{pmatrix}
\end{empheq}
In this representation the fibre above a given point $m\in M$ is no longer a straight circle but twists around the torus. This is depicted in Figure~\ref{figno:torus2}.

The equivariance equations still hold, but in this representation they do not express the rotation of the frame vector but rather the \enquote{twist} of the action of $\SO(2)$.
What we are interested in are spaces which \emph{locally} present the same structure as a frame bundle, thus we will focus on Equations \eqref{eqno:Zequiv}. We ask the following question: if we have a solid torus $P$ with a frame field $(e_1,e_2,\xi)$ satisfying Equations~\eqref{eqno:Zequiv}, can it always be identified with the frame bundle of a smooth manifold?


\subsubsection*{A twisted frame bundle}

\begin{wrapfigure}[24]{R}{0.4\textwidth}
	\centering
	\caption{The frame bundle of $\setR^2$ as a twisted torus}\label{figno:torus2}
	\includegraphics[width=0.4\textwidth]{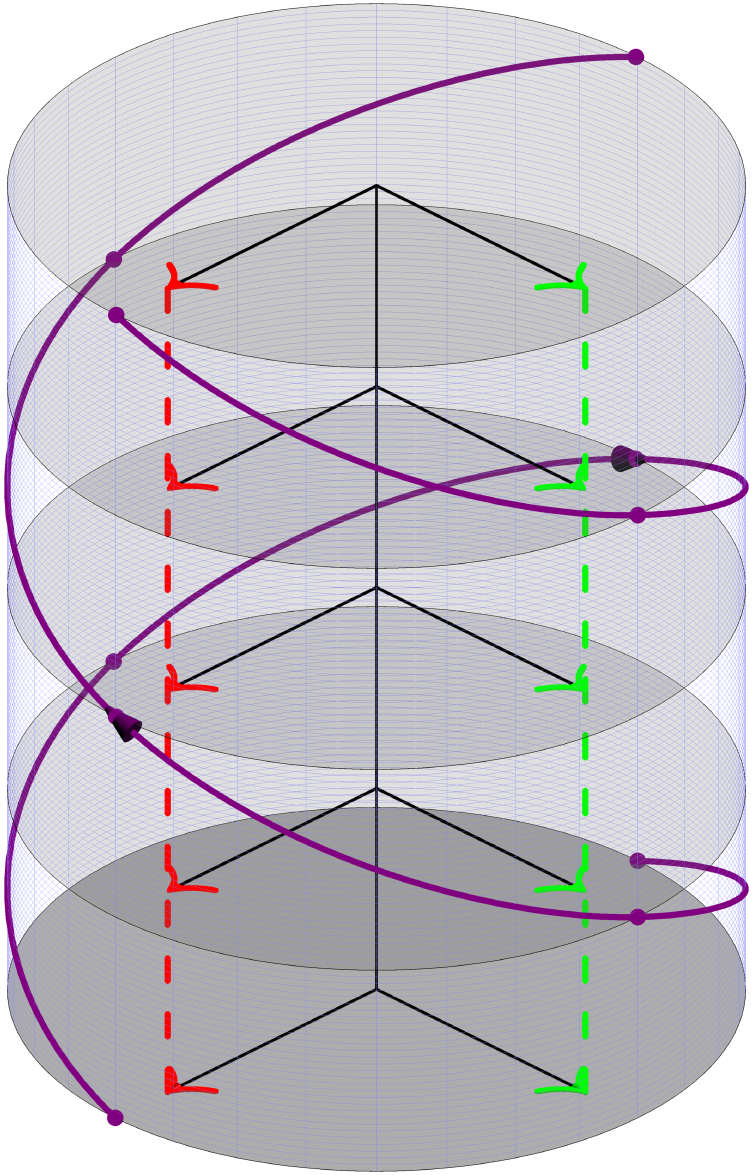}
\end{wrapfigure}

Counterexamples are readily found: it is possible that the action of $\xi$ does not integrate to an action of $\SO(2)$ which is free (all isotropy groups are trivial). More precisely it is possible for the orbits under the action of $\xi$ to have different lengths (finite or not). A simple example is depicted in Figure~\ref{figno:torus3} -- the top and the bottom face should be identified there as well.
It can be seen as a $(2+1)$-dimensional version of the M\"obius strip. Here infinitesimal equivariance holds with a factor%
\footnote{It may seem as though any real factor could be used but $\frac12$ is of specific relevance as it can be generalised to the \emph{projective} quotients of the higher dimension orthogonal groups. This is discussed in Section~\ref{secno:examples}}
of $\frac12$ for the generator $\xi$:
\begin{subequations}\label{eqno:Zequiv12}
\begin{empheq}[left=\empheqbiglbrace]{align*}
	[\xi,e_1] &= \frac12 e_2\\ 
	[\xi,e_2] &= -\frac12 e_1
\end{empheq}
\end{subequations}

The orbit of the origin of $M$ closes over one revolution, but the other orbits (in purple on Figure~\ref{figno:torus3}) require two revolutions to close. Although the action of $\xi$ can be integrated into a group action of $\setR$ on $M\times \SO(2)$, since points have different isotropy groups is it not possible to factor the action to a quotient $\SO(2)$ of $\setR$ such that all isotropy groups are trivial.

In this example if $P$ were to be interpreted as a frame bundle, it would be over its orbit space, which can be identified as $\setR^2/(x\sim -x)$. This is a singular space, with a conic singularity at the origin -- that is a manifestation of the varying size of the orbits. Therefore, in this example, $P$ equipped with the transverse vectors $(e_1,e_2)$ and the generator $\xi$, although it cannot be identified with the frame bundle of a smooth manifold, may still be interpreted as a frame bundle over a singular space.

This example is \enquote{proper} in a technical sense, which ensures the singularities are tame. One could study cases in which the \enquote{twist} of the orbits under $\xi$ is irrational so that the orbits do not close (except the central orbit) but are dense on the surface of smaller tori. These cases have harsher singularities and will not be studied in this paper.

In this paper we want to investigate such \enquote{generalised frame bundles} which have the same local structure as frame bundles but are not a priori actual frame bundles over smooth manifolds. Our study will generalise the case of the solid torus to a more general Lie group with a linear representation of any dimension. The local structure only defines the action of a Lie algebra, so we will have to attempt constructing an action of a Lie group starting from a Lie algebra action.

\subsubsection*{The structure of a connection}

\begin{wrapfigure}[22]{R}{0.4\textwidth}
	\centering
	\caption{The torus as a \enquote{twisted} frame bundle}\label{figno:torus3}
	\includegraphics[width=0.4\textwidth]{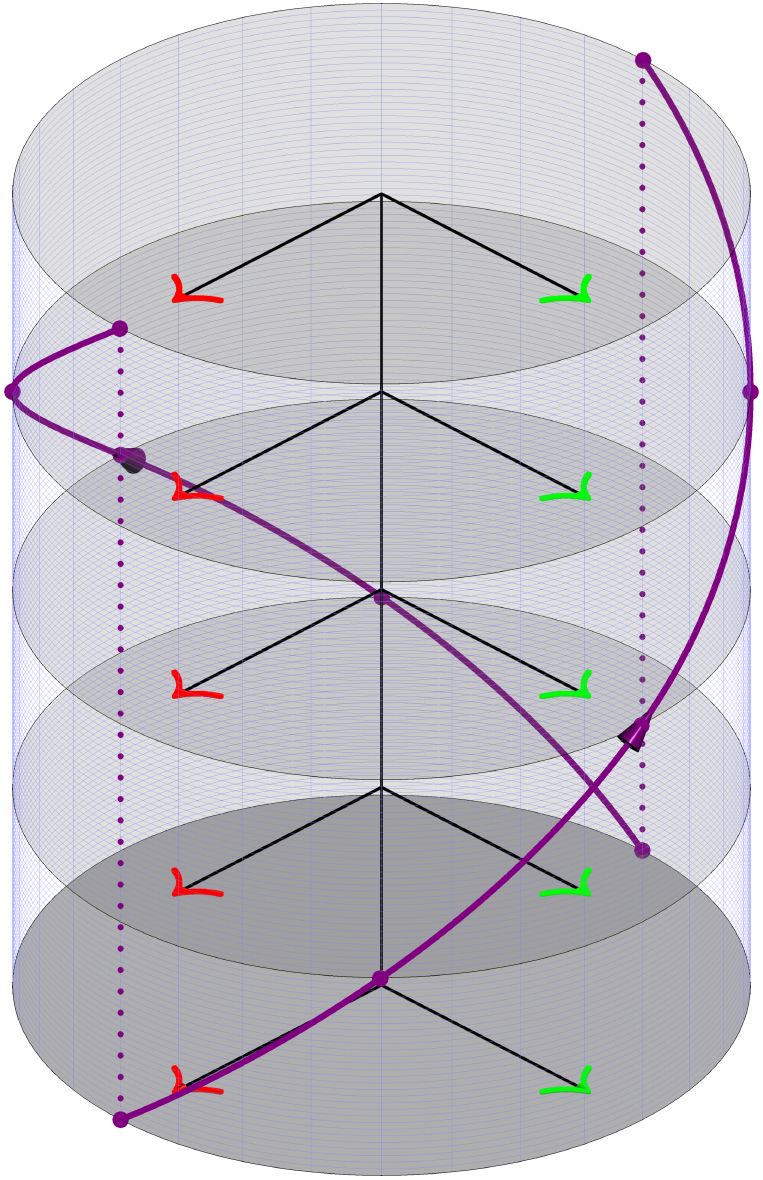}%
\end{wrapfigure}

The reader may have noticed that the trivialisation of the orthonormal frame bundle is not uniquely defined by the frame bundle structure. It is equivalent to the additional structure of a (metric and torsionless) \emph{flat} affine connection. The Euclidean plane inherits one induced by its affine space structure. In particular, the frame field which equips $P$ depends on this choice of connection. The structure we want to study on $P$ is more properly called \enquote{generalised frame bundle with connection} but we will use \enquote{generalised frame bundle} as a shorthand.

Indeed a more convenient way to encode the parallelism $(e_1,e_2,\xi)$ of $P$ is by using the dual coframe, which is a family of three $1$-forms $(\varpi^1,\varpi^2,\varpi^\xi)$. When $P$ forms an actual frame bundle, the forms $\varpi^1$ and $\varpi^2$ correspond to the so-called \enquote{solder form} of the frame bundle and the form $\varpi^\xi$ is a connection $1$-form. They are gathered in what is called a \emph{Cartan connection (1-)form}. Using this $3$-component $1$-form $\varpi$ is very convenient to formulate the infinitesimal equivariance equations~\eqref{eqno:Zequiv} (details in Section~\ref{secno:GFB}).

Because the frame bundle is the space supporting affine connections on the base manifold, this implies that covariant physical field theories can be formulated on the frame bundle. By extension, it is possible to formulate them on generalised frame bundles: they have all the required geometrical structure. This is of particular interest as a mean to formulate field theories on singular spaces.  For example, in higher dimension, given the Cartan connection form $\varpi$ it is very easy to construct the associated Riemann curvature tensor or the Einstein tensor, so that one can study Einstein's field equations (see the example in Section~\ref{secno:exdynGFB}).

\section{Generalising frame bundles}\label{secno:FB&GFB}
In this section, after a reminder of the structure of frame bundles with connection, we introduce the notion of generalised frame bundles.
They allow to define tensorial quantities in a similar way to frame bundles.

\subsection{Frame bundles and connections}

We start this section with general definitions about frame bundles and connections. We only give a succinct account of the structure, a detailed treatment can be found in general references such as~\cite{KobaNomi1}.
Let us recall the following notion which we will be using throughout.
\begin{definition}[Vertical vectors, horizontal forms]
Let $P\to M$ be a fibre bundle.
A vector (or a vector field) on $P$ is called \emph{vertical} if it projects to $0$ in $TM$.
A differential form on $P$ is called \emph{horizontal} if it has vanishing contraction with every vertical vector.
\end{definition}

We will also make recurrent uses of various notions of pullback:
\begin{definition}[Pullback]
Let $f : X \to Y$ be a smooth map.
\begin{enumerate}
	\item Given a vector-valued differential $k$-form $\alpha$ on $Y$, its pullback $f^*\alpha$ through $f$ is the following vector-valued differential $k$-form on $X$:
	\[
		\forall x\in X, \, v_1, \dots, v_k \in T_x X, \quad
		f^*\alpha(v_1, \dots, v_k) = \alpha(\d f(v_1), \dots \d f (v_k))
	\]

	\item Given a fibre bundle $P \xrightarrow{\pi} M$, the \emph{pullback bundle through $f$} is the fibred product 
	\[
		f^* P :=
			X\times_f P
				= \{ (x,p)\in X \times P \, | \, f(x) = \pi(p) \}
	\]
	It is a smooth manifold which is naturally a fibre bundle above $X$.

	\item When the above bundle $P$ is a principal $G$-bundle equipped with a principal connection form $\omega$,
	the pullback bundle $f^*P$ is a principal $G$-bundle and the pullback form $f^*\omega$ defines a principal connection form on $f^*P$.
	The corresponding connection is called the \emph{pullback connection} (through $f$).
\end{enumerate}
\end{definition}

\subsubsection{The frame bundle of a manifold}

Let $M$ be a smooth $n$-manifold. 
At each point $m\in M$ the tangent space $T_m M$ is a $n$-dimensional vector space and has a $n^2$-dimensional manifold of linear frames. They can be gathered into a fibre bundle over $M$ which is called the \emph{(linear) frame bundle of $M$} and which we write $\GL(M)$.

Since the set of frames of $T_m M$ has a natural right action of $\GL(n)$ which is both transitive and effective, $\GL(M)$ has the structure of a $\GL(n)$-principal bundle over $M$.
\com{Faut-il construire l'action ?}
Furthermore it is provided with a natural $\setR^n$-valued $1$-form which we now construct.

\subsubsection*{The canonical solder form}
Let $\pi : \GL(M)\to M$ be the fibration map. The bundle 
\[
	\pi^*TM \simeq \GL(M)\times_M TM
\]
is made of pairs $(p\in \GL(M), X\in T_{\pi(p)} M)$. 
Since $\GL(M)_m$ is the manifold of the linear isomorphisms of $T_m M$ with $\setR^n$, there is a natural \enquote{tautological} isomorphism of vector bundles
\[
	u :\pi^*TM \xrightarrow{\sim} \GL(M) \times \setR^n
\]
which converts tangent vectors to their coordinates in each frame.
Furthermore $u$ is equivariant under the right action of $\GL(n)$:
writing $R_g$ for the right action of $g\in\GL(n)$ on $\GL(M)$, the following diagram commutes
\[\begin{tikzcd}
	\pi^*TM 
		\ar[r,"u"]
		\ar[d,"{R_g}_*"]
		&
	\setR^n
		\ar[d,"g^{-1}"]
		\\
	\pi^*TM
		\ar[r,"u"]
		&
	\setR^n
\end{tikzcd}\]

The differential of $\pi$ can be represented as a map
\[
	\d \pi : T\GL(M) \to \pi^*TM
\]
so that composition gives a linear map
\[\begin{tikzcd}
	T\GL(M) 
		\ar[r,"\d\pi"']
		\ar[rr, bend left=20, "\theta"]
		&
	\pi^*TM
		\ar[r,"u"']
		&
	\setR^n
\end{tikzcd}\]

The fibration map $\pi$ is invariant under the action of $\GL(n)$: 
\[\pi\circ R_g = \pi\]
thus for all $p\in \GL(M)$, the following holds:
\[ \d\pi \circ \d R_g
	= {R_g}_*\d \pi \]
with ${R_g}_* : \pi^*TM \isom {R_g}_*\pi^*TM$ the natural identification.
Finally 
\[
	\theta \circ \d R_g|_{p}
	= u \circ \d \pi \circ \d R_g
	= u \circ {R_g}_* \d \pi
	= g \cdot u \circ \d \pi
	= g \cdot \theta
\]

The linear map $\theta$ can be understood as a (horizontal) $\setR$-valued $1$-form on $\GL(M)$ and is called the \emph{canonical solder form} of $\GL(M)$.
As we just proved, it is equivariant under the action of $\GL(n)$.

\subsubsection*{Vectors, tensor fields and differential forms}

Let $X$ be a vector field on $M$. It can be lifted to $\GL(M)$ into a section $\pi^*X$ of $\pi^*TM$ which is constant on the fibres of $\pi$. Applying $u$, the section $\pi^*X$ can be identified to a $\setR^n$-valued map, which is now $\GL(n)$-equivariant since $u$ is. Conversely an equivariant $\setR^n$-valued map on $\GL(M)$ can be identified to a $\GL(n)$-invariant $\pi^*TM$-valued map which can then be factored to a vector field on $TM$.
\begin{center}
	\begin{tikzcd}
		X \in \Gamma(M,TM)
			\ar[r,"\pi^*"]
			&
		\pi^*X \in \Gamma(\GL(M), \pi^*TM)
			\ar[r,"u"]
			\ar[d, "R_g^*"]
			&
		u(\pi^*X) \in \Gamma(\GL(M),\setR^n)
			\ar[d,"R_g^*"]
			\\
			&
		\pi^*X
		 &
		g^{-1}\cdot u(\pi^*X)
	\end{tikzcd}
\end{center}

This construction works not only for vectors but also for tensors living in any tensor product of $TM$ and $T^*M$, as well as tensor-valued horizontal differential forms. \com{Détailler pourquoi l'horizontalité est importante ?}

\begin{theorem}\label{thmno:Omhor}
Using a superscript $\GL(n)$ to denote $\GL(n)$-equivariant forms, there is an isomorphism
\begin{equation}
	\Omega^k_{hor}(\GL(M),\setR^n)^{\GL(n)} \simeq \Omega^k(M,TM)
\end{equation}\com{Est-il nécessaire d'écrire *ici* les formes différentielles ou peut-on simplement mentionner la généralisation aux formes différentielles plus loin ?}
and similarly when $TM$ is replaced with an associated tensor bundle.
\end{theorem}%
For example, this isomorphism maps the solder form $\theta$ to the identity of $TM$, represented as a $TM$-valued $1$-form.

Geometric structures can be defined on $M$ by specifying a restricted class of linear bundles. The corresponding structures, called $G$-structures, are presented in the next section.

\subsubsection{$G$-structures}\label{secno:Gstruct}

When $M$ has an orientation, it is possible to define a notion of \enquote{direct} frames. Conversely, given the class of direct frames, it is possible to identify a corresponding orientation. The same is true in pseudo-Riemannian geometry: it is equivalent to specify a metric or to specify the class of (pseudo)-orthonormal frames (that respect the ordered signature). The notion of $G$-structure sets the focus on the class of frames structured by the group relating these frames. More concretely: 

\begin{definition}[$G$-structure]\label{defno:Gstruct}
	Let $M$ be an $n$-manifold. Let $G$ be a Lie group \emph{equipped with an action on $\setR^n$}:
		\[ G\to \GL(n)  \]
	
	A \emph{$G$-structure} on $M$ is the data of a $G$-principal bundle $P\xrightarrow{\pi_P} M$ which has a \emph{$G$-equivariant bundle map} to $\GL(M)$, namely, the following diagram commutes:
	\[\begin{tikzcd}
		P \times G
			\ar[rr]
			\ar[d]
			&&
		\GL(M) \times \GL(n)
			\ar[d]
		\\
		P
			\ar[rr]
			\ar[rd]
			&&
			\GL(M)
			\ar[ld]
		\\
		&M&
	\end{tikzcd}\]
	In other words, a $G$-structure is a reduction of structure group of the linear frame bundle along $G\to \GL(n)$.
\end{definition}
Although the usual denomination is \enquote{$G$-structure}, the action of $G$ on $\setR^n$ is also an essential part of the defining data. As is standard, the action will often be left implicit. Note that we do not require the action does not have to be faithful, which will be especially relevant in Section~\ref{secno:PrinGstruc}. A general reference on $G$-structures is~\cite{TransformationGroups}.
	

\begin{example}
	$\GL(M)$ defines a canonical $\GL(n)$-structure on $M$, which is part of the differential manifold structure.
\end{example}
	
\subsubsection*{Solder forms}

The data of an equivariant map $P\to \GL(M)$ can be encoded in a $\setR^n$-valued $1$-form $\alpha$ on $P$ which satisfies the following two requirements:
\begin{itemize}
	\item The kernel is the vertical tangent bundle 
		\[VP := \ker (\d \pi_P : TP \to \pi_P^*TM) \subset TP \]
	\item It is equivariant under the action of $G$ on $\setR^n$.
\end{itemize}
Such a $1$-form is called a \emph{solder form} (sometimes \emph{soldering form}).

Stating that the kernel is $VP$ is equivalent to asserting that $\alpha$ factors to an injective bundle mapping
	\[ \pi_P^*TM \hookrightarrow P\times\setR^n \]
which is then necessarily an isomorphism because the fibers have the same dimension. $G$-equivariance of the $1$-form is equivalent to the $G$-equivariance of this mapping, so that such a $1$-form indeed gives a $G$-equivariant mapping 
\[P \to \GL(M) \]
It is then easily shown that the solder form on $P$ is the pull back of $\theta$ by the obtained map $P\to \GL(M)$, and conversely the pullback of $\theta$ by an equivariant bundle map is always a solder form on $P$.

Furthermore, $\alpha$ inducing an equivariant bundle isomorphism
\[
	\pi_P^*TM \isom P\times \setR^n
\]
implies that it further factors to a vector bundle isomorphism over $M$:
\[
	TM \simeq \pi_P^*TM/G \isom P\times_G \setR^n
\]
From this perspective, a solder form on $P$ is associated to a $P\times_G \setR^n$-valued $1$-form on $M$ which is \enquote{non-degenerate}.
Indeed this gives another characterisation of $G$-structures: they are $G$-principal bundles equipped with a vector bundle isomorphism
\[
	TM \simeq P\times_G \setR^n
\]

\subsubsection*{Associated tensor fields}

Tensor fields and tensor-valued differential forms can be defined using the principal bundle of a $G$-structure: by the same construction as for $\GL(M)$, $G$-equivariant $\setR^n$-valued forms can be identified with $G$-invariant $\pi_P^*TM$-valued forms. The same goes for the tensor bundles~\cite[Section 19]{TopicsDiffGeo}.

Furthemore, given any finite dimensional linear representation $V$ of $G$ it is possible to consider $G$-equivariant $V$-valued fields (and differential forms) on $P$. They are in natural bijection with the sections of the \emph{associated bundle} $P \times_G V$ on $M$
which is a vector bundle over $M$ with typical fibre $V$.

Similarly, there is an isomorphism between the spaces of vector-valued differential forms
\[
	\Omhor^k(P,V)^G \simeq \Omega^k(M,P[V])
\]
along the same lines as Theorem~\ref{thmno:Omhor}

\subsubsection{Principal connections on $G$-structures}

Let $M$ be a manifold with a $G$-structure given by a principal bundle $P\to M$ with a solder form $\alpha$.
The action of $G$ on $P$ induces an action of the associated Lie algebra $\Glie$ by differentiation of the $1$-parameter subgroups of diffeomorphisms. For $\xi\in\Glie$ we write $\bar\xi$ the corresponding vector field on $P$. Note that the $\bar\xi$ are vertical (they project to $0$ in $TM$) and they span all the vertical directions.

\begin{definition}[Principal connections]
	A \emph{$G$-principal connection $1$-form} on $P$ is given by a $\Glie$-valued $1$-form $\omega$ on $P$ such that:
	\begin{itemize}
		\item It is normalised for the action of $\Glie$:
			\[ \forall \xi\in\Glie, \qquad \omega(\bar\xi) = \xi \]
		\item It is $G$-equivariant as a $\Glie$-valued $1$-form: 
			\[ R_g^*\omega = \Ad_{g}^{-1} \omega  \]
	\end{itemize}
	The adjective \enquote{$G$-principal} will most often be omitted. We say that a connection $1$-form on $P$ defines a \emph{connection} on $P$.
\end{definition}

The kernel of a connection $1$-form is a supplementary subspace to the vertical tangent space and is called \emph{horizontal tangent space} to $P$, we will write it $HP$. A connection hence defines projections of a vector on $P$ to horizontal and vertical components.

\subsubsection*{Covariant derivative}

A principal connection $\omega$ allows defining derivatives of sections of associated bundles on $M$. Let $V$ be a representation of $G$ and $\psi\in \Gamma(P, V)^G$. The \emph{covariant derivative} of $\psi$ is defined as follows:
\[
	\dom \psi := \d \psi + \omega\cdot \psi\in \Omega^1(P,V)
\]
with $\omega \cdot \psi$ the $V$-valued $1$-form $X\mapsto \omega(X)\cdot \psi$. The $1$-form $\dom \psi$ is horizontal and equivariant.

This definition extends to $V$-valued equivariant forms: let $\psi \in \Omega^k_{hor}(P,V)^G$. We define the following product
\[\begin{tikzcd}
	\Omega^1(P,\Glie)\otimes \Omega^k(P,V)
		\ar[r] &
	\Omega^{1+k}(P,\Glie\otimes V)
		\ar[r] &
	\Omega^{k+1}(P,V)
		\\
	\omega \otimes \psi
			\ar[r, mapsto] &
	\omega\wedge \psi
			\ar[r, mapsto] &
	\omega \cdot \psi
\end{tikzcd}\]
It is used in the following definition of the \emph{covariant exterior derivative} of $\psi$:
\begin{equation}
	\dom \psi = \d \psi + \omega\cdot \psi
\end{equation}
It can be shown to correspond to the precomposition of $\d \psi$ with the projection \[\Lambda^{k+1} TP \to \Lambda^{k+1} HP \subset \Lambda^{k+1} TP \]

The $(k+1)$-form $\dom \psi$ is then an horizontal $V$-valued $(k+1)$-form which can be proved to be equivariant: as an element of $\Omega^{k+1}_{hor}(P,V)^G$ it is associated to a $P\times_G V$-valued $(k+1)$-form on $M$. In this way a connection allows defining covariant derivatives of sections of bundles on $M$, as summed up in the following diagram:
\[\begin{tikzcd}
	\Omhor^k(P, V)^G
		\ar[d, "\sim" sloped]
		\ar[r, "\dom"]
	&
	\Omhor^{k+1}(P, V)^G
		\ar[d, "\sim" sloped]
	\\
	\Omega^k(M, P\times_G V)
		\ar[r, "\d^\nabla"]
	&
	\Omega^{k+1}(M, P\times_G V)
\end{tikzcd}\]
Our focus will be on the fields defined over $P$ and not over $M$.

\subsubsection*{Curvature and torsion forms}

The obstruction to the integrability of the horizontal distribution $HP$ is contained in the \emph{curvature $2$-form} associated to the connection, defined as follows: 
\begin{equation}
	\Omega := \d \omega + \fwb{\omega}{\omega}
\end{equation}
with 
\[
	\wb{\beta}{\gamma}(X,Y) := [\beta(X),\gamma(Y)] - [\beta(Y),\gamma(X)]
\]

The curvature $2$-form is also involved in the square of the covariant exterior derivative operator through the following Bianchi-type identity: 
\[
	\forall\psi\in\Omega^k_{hor}(P,V)^G\qquad
			\dom\dom \psi = \Omega\cdot \psi
\]
using a similar notation to $\omega\cdot\psi$ for the product
\[
		\Omega^2(P,\Glie)\otimes \Omega^k(P,V) \to \Omega^{k+2}(P,V)
\]

Furthermore, the $G$-structure equips $P$ with a solder form $\alpha$. Its covariant exterior derivative is called the \emph{torsion} of $\omega$: 
\begin{equation}
	\Theta := \dom \alpha \in \Omhor^2(P,\setR^n)
\end{equation}
%
\subsubsection{Principal bundles on $G$-structures}\label{secno:PrinGstruc}
\com{Nommer la section produit de fibrés principaux ?}

Let $M$ be a manifold with a $G$-structure given by a principal bundle $P\to M$ equipped with a solder form $\alpha$. Let there be another principal bundle $Q\to M$ with a structure group $K$. We are looking for a way to characterise the structure of the principal bundle $Q$ as a structure on $P$.

The principal bundle $Q$ can be pulled back through $\pi_P$ to an $K$-principal bundle $\pi_P^*Q$ over $P$. The fact that the principal bundle over $P$ is a pullback bundle can be characterised by the structure of \emph{$G$-equivariant $K$-principal bundle}: the action of $G$ on $P$ lifts to morphisms of principal bundle of $\pi_P^*Q$ (which are identified with the identity on the spaces $Q_{\pi_P(p)}$). Namely, the following diagram commutes for every $g$ in $G$: 
\begin{equation*}
\begin{tikzcd}[column sep = huge]
	\pi_P^*Q
		\ar[r, "{
				(p,q)\mapsto (p\cdot g, q)
				}"]
		\ar[d]
	& \pi_P^* Q
		\ar[d]
	\\
	P
		\ar[r,"\cdot g"]
	&
	P
\end{tikzcd}
\end{equation*}

The structure of $\pi_P^*Q$ is clear when it is viewed as a fibred product of bundles: 
\[\begin{tikzcd}[column sep=6em, row sep=large]
	P\times_M Q \simeq \pi_P^*Q
		\ar[d,
			"\substack{
				K\text{-principal}\\
				G\text{-equivariant}
			}"',
			"\phi_P"
		]
		\ar[r,
			"\substack{
				G\text{-principal}\\
				K\text{-equivariant}
			}",
			"\phi_Q"'
		]
		\ar[rd, phantom, "\scalebox{2.5}{$\lrcorner$}", very near start]
	&Q
		\ar[d,"K\text{-principal}", "\pi_Q"']
	\\
	  P
	  	\ar[r,"G\text{-principal}"',"\pi_P"]
	& M
\end{tikzcd}\]
%
In particular the symmetry between the roles of $P$ and $Q$ is made manifest. The space $P\times_M Q$ has a structure of $K$-principal bundle over $P$, of $G$-principal bundle over $Q$ and of $G\times K$-principal bundle over $M$. In particular, the actions of $G$ and $K$ commute which implies that the bundle fibration over $Q$ is $K$-equivariant. 

Such structures have been explored in physics (see Kerner, Nikolova and Rizov's work \cite{TwoLevelKK}) and generalised in mathematical studies (see Lang, Li and Liu's paper \cite{DoubleBundles}).

\subsubsection*{Principal connections}

Assume now that $Q$ is equipped with a $K$-principal connection form $A\in\Omega^1(Q,\Klie)$. The connection pulls back under the projection $P\times_M Q \xrightarrow{\phi_Q} Q$ into an $K$-principal connection on $P \times_M Q$ which we write $\phi_Q^*A$ . Since $\phi_Q$ is invariant under precomposition by the action of elements of $G$,

\[ \phi_Q \circ R_g = \phi_Q 	\]
the same goes for $\phi_Q^*A$: 
\[
	R_g^*\phi_Q^*A = \phi_Q^* A
\]
Thus the pullback connection on $P\times_M Q$ is $G$-invariant, or $G$-equivariant for the trivial adjoint action of $G$ on $\Klie\subset \Glie\times \Klie$. Similarly, a principal connection on $P$ can be pulled back to $\PMQ$ to a $Q$-invariant connection.

Thus a principal connection on $P$ and a principal connection on $Q$ can be pulled back to $\PMQ$ and summed into a $\Glie\oplus \Klie$-valued $1$-form which is $G\times K$-equivariant and normalised on $\Glie\oplus \Klie$: this is exactly a $G\times K$-principal connection on $\PMQ$. Conversely, given a $G\times K$-principal connection on $\PMQ$, its respective $\Glie$ and $\Klie$ components are $G$-equivariant and $K$-invariant, respectively $K$-equivariant and $G$-equivariant, and induce principal connections on $P$ and $Q$. This is summed up in the following theorem:

\begin{theorem}
	Let $P\to M$ be a $G$-principal bundle and $Q\to M$ a $K$-principal bundle. There is a natural bijection between $G\times K$-principal connections on $\PMQ$ and couples of principal connections on $P$ and $Q$.
\end{theorem}

If $P$ is equipped with a solder form $\alpha$, it can also be pulled back to $\PMQ$ into a horizontal $1$-form which is equivariant under $G$ and invariant under $K$. Considering the trivial action of $K$ on $\setR^n$, $\phi_P^*\alpha$ is equivariant under $G\times K$. Furthermore its kernel is made of the vectors of $\PMQ$ mapped to vertical vectors on $P$: this is exactly the vertical vectors of $\PMQ\to M$. Thus we conclude: 
\begin{theorem}
	If $P\to M$ with a solder form $\alpha$ defines a $(G\to\GL(n))$-structure and $Q\xrightarrow{\pi_Q} M$ is an $K$-principal bundle, then $\phi_Q^*\alpha$ defines a solder form on $\PMQ\to M$ for the representation $G\times K \xrightarrow{p_G} G \to \GL(n)$.
\end{theorem}

More generally, any representation $V$ of $G$ or $K$ can be extended to a representation of $G\times K$ by a trivial action of the other factor. An equivariant $V$-valued field on $P$ or $Q$ can then be pulled back to $\PMQ$ to an equivariant $V$-valued field. Conversely, when $K$ (resp. $G$) acts trivially on $V$, every $G\times K$-equivariant field on $\PMQ$ factors to a $G$-equivariant field on $P$ (resp. a $K$-equivariant field on $Q$).

\subsection{Generalised frame bundles}\label{secno:GFB}

\subsubsection{Definition}

Let $P\to M$ be an $G$-principal bundle equipped with a solder form $\alpha$ and a connection form $\omega$.
Gathering the two forms in a $\Glie\oplus\setR^n$-valued $1$-form gives a \emph{coframe} on $P$, which we write $\varpi:=\omega\oplus \alpha$.
The coframe $\varpi$ is $G$-equivariant.

As a coframe, $\varpi$ can be used to find the fundamental vector fields of the action of $\Glie$: if $\xi\in\Glie$ is represented by $\bar\xi$, the following holds: 
\[
	\bar\xi = \varpi^{-1}(\xi,0)
\]


The idea of \emph{generalised frame bundles with connection} is to start from the manifold $P$ and the object $\varpi$ and use it to define the action of $\Glie$ and the vertical and horizontal tangent spaces. The base manifold $M$ is to be reconstructed as an orbit space. Of course, $\varpi$ cannot be any arbitrary coframe: it has to be $G$-equivariant, and the vector fields $\bar\xi$ have to form an action of $\Glie$. As we will see, equivariance is already a stronger requirement than compatibility with the bracket. The coframe $\varpi$ only defines an action of the Lie algebra $\Glie$ (and not a group action) so we will be looking at equivariance under the Lie algebra action. 

Let us now derive an intrinsic characterisation of the $\Glie$-equivariance of $\varpi$.

\subsubsection*{Cartan $1$-forms}

First we equip $\Glie\oplus\setR^n$ with the Lie algebra structure $\Glie\ltimes \setR^n$ so that the action of $\Glie$ on $\Glie$ and $\setR^n$ can be handled in a uniform fashion as an adjoint action. We will use indices $i,j\dots$ for $\Glie$, indices $a,b\dots$ for $\setR^n$ and indices $A,B\dots$ for $\Glie\ltimes \setR^n$. 
Equivariance under $\Glie$ can be stated as:
\begin{equation}\label{eqno:equivLie}
	\Lie_{\bar\xi} \varpi + \ad_\xi \varpi = 0
\end{equation}
The left hand term can be reformulated as follows:
\begin{equation*}
	\Lie_{\bar\xi} \varpi + \ad_\xi \varpi
		= (i_{\bar\xi} \d + \d i_{\bar\xi} )\varpi + [\xi,\varpi]
		= i_{\bar\xi} \d\varpi + \d \xi + [\varpi(\bar\xi),\varpi]
		= i_{\bar\xi} \lp \d\varpi + \fwb\varpi\varpi \rp 
\end{equation*}
Thus Equation~\eqref{eqno:equivLie} can be reformulated as 
\begin{equation}\label{eqno:equivcontr}
	\forall \xi\in\Glie,\qquad
		i_{\bar\xi} \lp \d\varpi + \fwb\varpi\varpi \rp  = 0
\end{equation}

Since the vectors $(\bh)_{\h\in\Glie}$ span the vertical directions and the forms $\alpha^a$ form a basis of the horizontal $1$-forms, this is also equivalent to the existence of variables coefficients $\Omega^a_{bc}$, $\Omega^i_{bc}$ such that
\begin{subequations}\label{subeqno:equivalal}
\begin{align}
	\d\omega^i + \fwb\omega\omega^i &= \frac12\Omega^i_{bc}\alpha^b\wedge\alpha^c\\
	\d\alpha^a + \wb\omega\alpha^i  &= \frac12\Omega^a_{bc}\alpha^b\wedge\alpha^c
\end{align}\end{subequations}

Another reformulation of \eqref{eqno:equivcontr} is obtained by contraction with $\bar\zeta :=\varpi^{-1}(\zeta)$ for $\zeta=(\zeta_\Glie,\zeta_{\setR^n})\in \Glie\ltimes \setR^n$: 
\begin{equation}\label{eqno:CartanFormBracket}
	\forall \xi\in\Glie, \ \forall \zeta\in\Glie\ltimes \setR^n, \quad 
	\d\varpi(\bar\xi,\bar\zeta) + [\varpi(\bar\xi),\varpi(\bar\zeta)] = 0
\end{equation}
Since $\varpi(\bar\zeta) = \zeta$ is a constant $\Glie\ltimes \setR^n$-valued field, Equation~\eqref{eqno:CartanFormBracket} can be rephrased as
\begin{equation}\label{eqno:equivbracket}
	[\bar\xi,\bar\zeta] = \overline{[\xi,\zeta]}
\end{equation}
Thus Equations~\eqref{subeqno:equivalal} are equivalent to a bracket compatibility condition on the vector fields $\bar\zeta$ which in particular implies that the vector fields $\bar\xi$ define a free action of $\Glie$ on $P$.

This motivates the following definition: 
\begin{definition}[Cartan $1$-form]
	Let $P$ be a manifold and $\Glie\ltimes\setR^n$ a semi-direct product Lie algebra of the same dimension as $P$.
	
	A $\Glie\ltimes\setR^n$-valued \emph{Cartan $1$-form} (or Cartan form) on $P$ is a $\Glie\ltimes\setR^n$-valued coframe $\varpi^A$ such that there exists (variable) coefficients $\Omega^A_{bc}$ such that:
	\begin{equation}\label{eqno:EqCartanalpha}
		\d \varpi^A + \fwb\varpi\varpi^A = \frac12\Omega^A_{bc}\alpha^b\wedge \alpha^c
	\end{equation}
%
\end{definition}

The notion was already defined by Alekseevsky and Michor~\cite{DiffGeoCartan} under the name \enquote{Cartan connection}; they propose a notion of \enquote{generalized Cartan connection} that allows for a degenerate $1$-form $\varpi$ but in exchange requires a pre-existing action of the Lie algebra. Indeed a good part of our formal manipulations will not require vector fields and will apply to degenerate forms as well but in this case $\varpi$ is not sufficient to define the action of $\Glie$, which we want to be able to do. 
Cartan $1$-forms are also a case of what would be called $\Glie$-flatness of $\varpi$ by Gielen and Wise~\cite{ObsSpace} (with $K=1$).

As justified above, a Cartan $1$-form defines on the manifold an action of $\Glie$ for which it is equivariant. A familiar example of the phenomenon is the case of \emph{Maurer-Cartan forms}:
\begin{example}[Maurer-Cartan forms of Lie groups]\label{exno:MCLie}
	Let $G$ be a Lie group and $\Glie$ its Lie algebra, identified with the tangent space at identity $T_e G$.
	
	Write $L_g$ for the left translation map by any element $g$ of $G$. The following map defines a parallelism: 
	\[
		\omega_G : \begin{cases}
			TG\to T_e G \xrightarrow{\sim} \Glie \\
			(g,X)\mapsto  (g,(\d L_g)|_{e}^{-1} X)
		\end{cases}
	\]
	
	The map $\omega_G$ can be seen as a $\Glie$-valued $1$-form on $G$: $\omega_G\in \Omega^1(G,\Glie)$. It is called the \emph{Maurer-Cartan form} of $G$ and satisfies the following \emph{Maurer-Cartan equation}: 
	\[
		\d \omega_G + \fwb{\omega_G}{\omega_G} = 0
	\]
	
	Vector fields with a constant image by $\omega_G$ are the \emph{left-invariant vector fields} on $G$. They are the vector fields by which $\Glie$ acts on $G$ \emph{on the right}.
	
	Correspondingly, it is possible to define a \enquote{left} Maurer-Cartan form $\omega_G^L$ on $G$ using the right action of $G$ (the previous one can be called right Maurer-Cartan form). It satisfies a different\footnote{It can be seen as the usual Maurer-Cartan equation for the opposite Lie bracket.}
	Maurer-Cartan equation:
	\[
		\d \omega_G^L - \fwb{\omega^L_G}{\omega^L_G} = 0
	\]
\end{example}

\begin{example}[Maurer-Cartan forms on manifolds]\label{exno:MC}
	Let $P$ be an $m$-dimensional manifold and $\Glie$ an $m$-dimensional Lie algebra. A $\Glie$-valued Maurer-Cartan form on $P$ is a $\Glie$-valued coframe $\varpi$ satisfying the following Maurer-Cartan equation: 
	\begin{equation}
		\d\varpi + \fwb\varpi\varpi = 0
	\end{equation}
	It can be interpreted as a $\Glie\ltimes \setR^0$-valued Cartan $1$-form on $P$.
	
	Such a coframe defines a \emph{transitive} action of $\Glie$ on $P$, and conversely transitive Lie algebra actions are associated to such coframes.
\end{example}

\begin{remark}[Curvature form and symmetry breaking]
	As is suggested by the Maurer-Cartan forms, the curvature form $\d\varpi + \fwb\varpi\varpi$ quantifies how far the vector fields $\bz$ are from forming a representation of the Lie algebra $\Glie\ltimes\setR^n$.
	
	The structure discussed here can be approached starting from $\Omega$ rather than the semi-direct product $\Glie\ltimes\setR^n$: we have an abstract Lie algebra $\Hlie$ and we are looking for a splitting $\Hlie = \Glie\ltimes \mathfrak a$ such that $\Omega$ is horizontal in the sense of having only components along the $\mathfrak a$ directions. From this perspective, put forward in Wise's works~\cite{SymBrGrav, MMGravCartan, ObsSpace}, the corresponding \emph{Cartan geometry} can be understood as geometry with symmetry broken down from $\Hlie$ to $\Glie$.
\end{remark}

We will call a manifold equipped with a $\Glie\ltimes\setR^n$-valued Cartan $1$-form a \emph{generalised frame bundle with connection} modelled over $\Glie\ltimes\setR^n$, often shortened to \emph{generalised frame bundle}. One benefit of using the structure of generalised frame bundle is that it is a purely local structure. This fact is used in a recent paper by the author~\cite{GFB}
to obtain this structure from solutions to differential equations coming from a variational principle. We come back to this situation in Section~\ref{secno:exdynGFB}.
We have justified the following result: 
\begin{theorem}\label{thmno:CartformAct}
	Let $P$ be a generalised frame bundle with a $\Glie\ltimes\setR^n$-valued Cartan $1$-form $\varpi$. Then the map
	\[
		\h\in \Glie\mapsto \varpi^{-1}(\h,0) \in \Gamma(TP)
	\]
	defines a free action of the Lie algebra $\Glie$ on $P$ for which $\varpi$ is an equivariant $\Glie\ltimes\setR^n$-valued $1$-form.
\end{theorem}
 
We now extend the constructions used in the previous sections to the case of generalised frame bundles. 

\subsubsection{Basic fields on generalised frame Bundles}

Let $P$ be a generalised frame bundle with Cartan $1$-form $\varpi=\omega\oplus\alpha$. The coframe defines an $\Glie$-equivariant identification 
\[
	TP\simeq P\times \Glie\ltimes \setR^n
\]
We define the horizontal and vertical distributions:
\begin{align}
	HP &:= \ker \omega = \varpi^{-1}(\Glie\oplus 0)\\
	VP &:= \ker \alpha = \varpi^{-1}(0 \oplus \setR^n)
\end{align}

Vectors will be called \emph{horizontal} if they belong to $HP$ and vertical if they belong to $VP$. Conversely, differential forms will be called \emph{horizontal} if they have vanishing contraction with vectors of $VP$ and \emph{vertical} if they have vanishing contraction with vectors of $HP$. The space of horizontal differential forms on $P$ will be written $\Omhor^\bullet(P)$.

The structure equation
\begin{equation*}
	\d \alpha^a + \wb\varpi\alpha^a = \frac12\Omega^a_{bc}\alpha^b\wedge \alpha^c
\end{equation*}
implies that the ideal spanned by the forms $(\alpha^a)$ is a differential ideal and as a consequence the vertical distribution $VP$ is involutive: it defines a foliation on $P$, which is regular.

The Lie algebra $\Glie$ acts on $P$ and the action naturally lifts to natural vector bundles such as $TP$ or $\Lambda^\bullet TP$. The action on $T^*P$ preserves the horizontal and vertical forms: 
\[
	\Lie_{\bar\xi} \omega = -\ad_\xi \omega
\]
is a vertical form and 
\[
	\Lie_{\bar\xi} \alpha = -\ad_\xi \alpha
\]
is a horizontal form, for all $\xi\in\Glie$, since the action of $\Glie$ preserves the decomposition $\Glie \oplus \setR^n$.

As a consequence, $\Glie$ also preserves the verticality or horizontality of vectors.

\begin{definition}[Basic fields]\label{defno:Basicfields}
We define (local) \emph{basic vector fields} on $P$ as local vector fields which are \emph{horizontal} and \emph{$\Glie$-equivariant}.
Note that a horizontal vector field is identified through $\alpha$ to a $\setR^n$-valued field. Since $\alpha$ is equivariant, it is equivalent to ask for the horizontal vector field to be $\Glie$-invariant or for the $\setR^n$-valued field to be $\Glie$-equivariant.

Similarly we define \emph{basic tensor fields}, respectively \emph{basic differential forms}, as fields with value in a tensor product of $\setR^n$ and ${\setR^n}^*$, resp. \emph{horizontal} differential forms, which are $\Glie$-equivariant. Similarly to the situation of standard frame bundles, a horizontal differential form can be identified with a $\Lambda^\bullet {\setR^n}^*$-valued field.

Finally, given a representation $V$ of $\Glie$, a \emph{basic section} of $V$ is a $\Glie$-equivariant $V$-valued field. They generalise the sections of associated bundles.
We denote as follows the space of basic $V$-valued differential forms:
\begin{equation}
	\Omega^\bullet_{bas}(P,V) := \Omega^\bullet_{hor}(P,V)^\Glie
\end{equation}
\end{definition}

\subsubsection*{Covariant derivation}

\begin{definition}[Covariant derivative]
Let $V$ be a representation of $\Glie$ and $\psi$ a basic section of $V$. Its \emph{covariant differential} is defined as
\begin{equation}
	\dom \psi := \d\psi + \omega \cdot \psi \in \Omega^1_{bas}(P,V)
\end{equation}
which is a basic $V$-valued $1$-form. More generally for a basic $V$-valued $k$-form $\psi$, its covariant exterior differential is defined as
\begin{equation}
	\dom \psi := \d\psi + \omega\cdot \psi \in \Omega^{k+1}_{bas}(P,V)
\end{equation}
which is a basic $V$-valued $(k+1)$-form.
\end{definition}
It is important to note that in spite of the natural isomorphism $\Ombas^\bullet(P) \isom \Gamma \lp \ExT^\bullet {\setR^n}^* \rp$, the respective operators of exterior covariant differential on $\Ombas^\bullet(P)$ and $\Gamma \lp \ExT^\bullet {\setR^n}^* \rp$ are different.

\subsubsection{Curvature and torsion forms}

The \emph{curvature $2$-form} of the Cartan $1$-form is defined similarly to the frame bundle case:
\begin{equation}
	\Omega := \d \omega + \fwb\omega\omega \in \Omega^2_{bas}(P,\Glie)
\end{equation}
The \emph{torsion} is defined in a similar fashion: 
\begin{equation}
	\Theta := \d\alpha + \wb\omega\alpha \in \Omega^2_{bas}(P,\setR^n)
\end{equation}
Let $\psi$ be a basic $V$-valued $k$-form. The following Ricci-type identity holds: 
\begin{equation}
	\dom\dom \psi = \Omega\cdot \psi
\end{equation} 
The curvature satisfies a Bianchi identity:
\begin{equation}
	\dom\Omega = 0
\end{equation}

\subsubsection{Principal bundles}\label{secno:CartGeoPrin}
Let $P$ be a generalised frame bundle with a Cartan form $\varpi$. 
Given a Lie group $H$, we want to generalise the structure of $H$-principal bundles to the framework of generalised frame bundles. More precisely, we look to define a structure which will generalise the fibre product of a frame bundle and a $H$-principal bundle equipped with a connection. There are two approaches.

The first is to generalise $G$-equivariant $H$-principal bundles on $P$. Of course, $P$ does not have a group action but a Lie algebra action of $\Glie$ so that we can only consider $\Glie$-equivariant bundles.

\begin{definition}[$\Glie$-equivariant $H$-principal bundles]
	 An \emph{$\Glie$-equivariant $H$-principal bundle} on $P$ is a fibre bundle $E\xrightarrow{\phi} P$ with a lift of the action of $\Glie$ to $H$-invariant vector fields on $E$.
\end{definition}

Let $E$ be such a fibre bundle. We shall assume it is provided with an $\Glie$-invariant $H$-principal connection $1$-form which we call $A$ (which is always possible for a pullback principal bundle as described in Section~\ref{secno:PrinGstruc}). The connection form $A$ defines \emph{on $E$} horizontal and vertical distributions; $A$ defines a $\Hlie$-valued coframe on the vertical distribution, while the Cartan form $\varpi$ pulls back to a coframe $\phi^*\varpi$ of the horizontal distribution of $E$. They can be gathered into a $\Hlie\oplus \Glie$-valued coframe on $E$:
\[
	A\oplus \phi^*\varpi \in \Omega^1(E,\Hlie\oplus\Glie)
\]
which is both $H$-equivariant and $\Glie$-equivariant.

For $\xi\in\Glie$ we define
\[
	\hat\xi :=  \lp A \oplus \phi^*\varpi \rp^{-1}(0\oplus \xi)
\]
which represents $\xi$ by a horizontal vector field on $E$. The invariance of $A$ under the fields $\hat\xi$ for $\xi\in\Glie$ is formulated as
\[
	0 = \Lie_{\hat\xi} A
		= \lp \d i_{\hat\xi} + i_{\hat\xi} \d \rp A
		= \d 0 + i_{\hat\xi} \d A
\]
Since $A(\hat\xi)=0$, this is equivalent to
\[
	i_{\hat\xi} \lp \d A + \fwb A A \rp = 0
\]

We also define for $u\in\Hlie$ the vector fields $\hat u$ by which $\Hlie$ acts on $E$. Similarly to Equation \eqref{eqno:equivcontr}, equivariance of $A$ is equivalent to
\[
	\forall u\in \Hlie, \ 
		i_{\hat u} \lp \d A + \fwb A A \rp = 0
\]

The conclusion is that the equivariance of $A$ under $\Hlie\oplus\Glie$ is equivalent to the existence of coefficients $F^I_{ab}$, with $I$ superscripts associated to the space $\Hlie$, such that
\begin{equation}
	\d A^I + \fwb A A^I = \frac12 F^I_{bc} \phi^* \lp \alpha^b\wedge\alpha^c \rp
\end{equation}

The $\Glie$-equivariant $H$-principal connection thus satisfies an equation similar to~\eqref{eqno:EqCartanalpha}, and the $\Hlie\oplus\Glie$-valued coframe $A\oplus \phi^*\varpi$ is a Cartan $1$-form: using a superscript $B$ for $\Hlie\oplus\Glie$, we have
\begin{equation}
	\d (A\oplus\phi^*\varpi)^B + \fwb{A \oplus \phi^* \varpi}{A\oplus \phi^*\varpi}^B
	= \frac12 \Omega^B_{bc} \phi^* \lp \alpha^b\wedge\alpha^c \rp
\end{equation}

To sum up, the $\Glie$-equivariant $H$-principal bundle $E$ has the structure of a generalised frame bundle modelled on $\Glie\ltimes\setR^n$.

This brings us to the second approach, which is to consider the $H$-principal fibration as directions in a generalised bundle taking the place of the bundle $E$ above, instead of an actual $H$-principal bundle. This requires the generalised bundle to have the data of a $H$-principal connection which will be integrated into a coframe used to define the generalised frame bundle structure. A generalised frame bundle modelled on $\Hlie\oplus\Glie\ltimes\setR^n$ is equipped with commuting actions of $\Hlie$ and $\Glie$. Since $\Hlie$ only acts trivially on $\setR^n$, one can define basic sections of vector bundle associated to representations of $\Glie$. In this sense, such a generalised frame bundle can support \enquote{internal degrees of freedom} with infinitesimal directions corresponding to $\Hlie$.

From this perspective, principal bundles are readily integrated in the formalism of generalised frame bundles: they correspond to a Lie-subalgebra direct factor $\Hlie\subset \Hlie\oplus\Glie\ltimes\setR^n$ which acts trivially on $\setR^n$.

\subsection{Examples of generalised frame bundles}\label{secno:examples}
In this section we list a few examples in which the structure of generalised frame bundle with connection is relevant.
A first class of examples actually have a group action but with varying orbit types.

\subsubsection{Conical singularity with restricted chirality}\label{secno:exspinchir}

Consider the group $\Spin(4)$. It can be decomposed as a direct product of groups:
\[
	\Spin(4) \simeq \Spin(3)\times \Spin(3) \simeq \SpH(1)\times \SpH(1)
\]
with $\SpH(1)$ the group of unitary quaternions. The action of $\Spin(4)$ on $\setR^4$ by projection to $\SO(4)$ can be represented using the quaternionic structure: $\SpH(1)\times \SpH(1)$ acts on the quaternion space $\setH$ by:
\begin{equation}\label{eqno:ActSpinH}
\begin{alignedat}{3}
	\Spin(4)\times \setH &\simeq\SpH(1)\times \SpH(1)\times \setH &&\to \setH\\
	(g,z) &\equiv(p,q,z) &&\mapsto \bar{p} z q
\end{alignedat}
\end{equation}
Indeed the Clifford algebra $\Cl(4)$ is isomorphic to $\Mat_2(\setH)$ and the subspace of vectors $\setR^4$ can be identified with a quaternionic line within $\Cl(4)$ on which $\Spin(4)$ has the described action. 

Consider now the semi-direct product $\Spin(4)\ltimes \setR^4$: it is the product manifold equipped with the following product structure:
\[
	(g_1,x_1)\cdot (g_2,x_2) = (g_1g_2, \ g_2^{-1}\cdot x_1 + x_2) 
\]
It can be interpreted as the space of \enquote{spinorial frames} above the affine (coset) space 
\begin{align*}
	\lp \Spin(4)\ltimes \setR^4 \rp / \Spin(4) &\simeq \setR^4\\
		[g,x] &\mapsto g\cdot x
\end{align*}

Let $\gamma^5$ be a fixed chirality element of $\Spin(4)$:
it is the (ordered) product in the Clifford algebra $\Cl(4)$ of vectors of an orthonormal direct basis (its sign depends on the chosen orientation). It squares to identity and takes the form
\begin{align*}
	\Spin(4) &\simeq \SpH(1)\times \SpH(1)\\
		\gamma^5 &\equiv (1,-1)
\end{align*}
The chirality element generates a central $\Zdeux$ subgroup of $\Spin(4)$:
\[
	\{(1,1),(1,-1)\} \subset \SpH(1)\times \SpH(1) \simeq \Spin(4)
\]

Let us consider the quotient of $\Spin(4)\ltimes \setR^4$ under the corresponding left action of $\Zdeux$: 
\[
	P:= \lp \Zdeux \rp \backslash \lp \Spin(4)\ltimes \setR^4 \rp
\]
with the generator of $\Zdeux$ acting as
\[
	\gamma^5\cdot (g,x) = (\gamma^5 g, x)
\]
The quotient map
\[
	\Spin(4)\ltimes \setR^4  \to
	\lp \Zdeux \rp \backslash \lp \Spin(4)\ltimes \setR^4 \rp
\]
is a $2$-fold covering map.
The action of $\Zdeux$ on the left on $\Spin(4)\ltimes \setR^4$ factors to the quotient $\setR^4$ on which it acts by parity, according to Equation~\eqref{eqno:ActSpinH}:
\[\begin{tikzcd}
	(g,x)\in \Spin(4)\ltimes \setR^4
		\ar[d, mapsto]
		\ar[r, "\gamma^5", mapsto]
	&
	(\gamma^5 g, x) \in \Spin(4)\ltimes \setR^4
		\ar[d, mapsto]
	\\
	(g\cdot x) \in \setR^4
		\ar[r, "\gamma^5", dashed, mapsto]
	&
	\gamma^4\cdot g \cdot x \in \setR^4
\end{tikzcd}\]

As a subgroup of $\Spin(4)\ltimes \setR^4$, $\{1, \gamma^5\}$ is no longer central and is not even a normal subgroup so that $P$ is not a quotient group but it still has a left action of $\Spin(4)$.

The Maurer-Cartan form $\varpi$ on $\Spin(4)\ltimes \setR^4$ is invariant under the left action of $\Spin(4)\ltimes \setR^4$ hence factors to $P$. The manifold $P$ is naturally fibred above $\lp \setZ/2\setZ \rp \backslash \setR^4$ which has a conical singularity at the origin.

The perspective of generalised frame bundles suggests defining spinor fields on $\lp \setZ/2\setZ \rp \backslash \setR^4$ as basic spinor fields on $\lp \Zdeux\rp \backslash(\Spin(4)\ltimes\setR^4)$, according to Definition~\ref{defno:Basicfields}.

Computing the isotropy groups, one finds that for $x$ in $\setR^4\setminus\{0\}$, $[e,x]\in P$ has orbits of type $\Spin(4)$ but $(e,0)$ has orbit type $\Spin(3)\times\SO(3)$. That is, the $\Spin(4)$-isotropy groups of the orbit of $[e,0]$ correspond to the subgroup $1\times \Zdeux \subset \Spin(3)\times\Spin(3)$. The decomposition $\Spin(4) \simeq \Spin(3)\times \Spin(3)$ corresponds to the decomposition of Dirac spinors into left-handed and right-handed spinors. For a spinor, being invariant under $1\times \Zdeux$ is equivalent to having a vanishing right-handed part.
As a consequence, basic spinor fields on $\lp \Zdeux \rp \backslash(\Spin(4)\ltimes\setR^4)$ necessarily have a vanishing right-handed component above the origin of $(\Zdeux)\backslash \setR^4$.

In other words, the geometry of the generalised frame bundle $P$ requires the spinor fields on $(\Zdeux) \backslash \setR^4$ to have a right-handed part that vanishes at the origin.
%
%

\subsubsection{Locally Klein geometries}

This example generalises the previous one and gives a large family of \emph{flat} frame bundles above singular spaces.

Let $G$ be a Lie group acting on a vector space $V$. Consider the Lie group $G\ltimes V$: it has a (right) Maurer-Cartan form $\varpi$ (Example~\ref{exno:MC}). The action of the group $G$ \emph{on the left} on $G\ltimes V$ leaves $\varpi$ invariant while it is equivariant under the right action of $G$.

The group $G\ltimes V$ can be seen as a frame bundle as follows:
\[\begin{tikzcd}[row sep = tiny]
	G \ltimes V 
		\ar[r, twoheadrightarrow]
	&
	\lp G\ltimes V \rp /G 
		\ar[r,"\sim"]
	& V
	\\
	(g,v)
		\ar[r,mapsto]
	&
	{[g,v]} 
		\ar[r,mapsto]
	& g\cdot v
\end{tikzcd}\]
The \emph{diffeomorphism} $\lp G\ltimes V \rp /G \simeq V$ comes from following section: 
\[
	v \in V \mapsto (e,v) \in G\ltimes V
\]
with $\{e\}\times V$ crossing exactly once each (right) orbit under $G$. Since $G$ is not a normal subgroup, the isomorphism $\lp G\ltimes V \rp /G \simeq V$ is not a group homomorphism. Furthermore the splitting gives a section of the frame bundle over $V$ so that the frame bundle is trivialisable:
\[
	\begin{tikzcd}
	G\ltimes V
		\ar[rd, "{(g,v)\mapsto g\cdot v}"']
		\ar[rr,"\sim"', "{(g,v)\mapsto (g,g \cdot v)}"]
	&&
	G\times V
		\ar[ld, "{(g, v')\mapsto v'}"]
	\\
	& V &
	\end{tikzcd}
\]
with the principal action of $G$ on $G\times V$ being trivial:
\[
	(g,v) \cdot g_1 = (gg_1, v)
\]
On the trivialised bundle, the Maurer-Cartan form takes the form
\[
	\omega_G \oplus g^{-1}\cdot \omega_V
\]

Now let $K$ be a discrete subgroup of $G$. It defines a left coset manifold
\[
	P := K\backslash \lp G\ltimes V \rp
\]
of which $G\ltimes V$ is a covering. Furthermore the left action of $K$ commutes with the right action of $G$ so that $G$ has an induced right action on $P$.

Since $\varpi$ is invariant under the left action of $G$, it factors to $P$ to a $1$-form $\varpi_P$ which still satisfies the Maurer-Cartan equation:
\[
	\d\varpi_P + \fwb{\varpi_P}{\varpi_P} = 0
\]

To understand the underlying manifold, let us construct the orbit space of $P$ under the right action of $P$: it is the double coset space
\[
	K\backslash \lp G\ltimes V \rp /G
\]
Since we know $\lp G\ltimes V \rp /G$ is isomorphic to $V$ \emph{as a left $G$-manifold}, we can conclude the following:
\[
	K\backslash \lp G\ltimes V \rp /G \simeq K\backslash V
\]

As a conclusion, $K\backslash \lp G\ltimes V \rp$ models a $G$-structure above the quotient space $K\backslash V$.
These examples are generalisations of what Sharpe~\cite{Sharpe} calls \emph{locally Klein geometries}: they are (connected, regular) quotients of homogeneous spaces (here the affine space $V$) by discrete groups.

The example presented in the introductory section \ref{secno:toymodel} can be interpreted as a case of this construction.
The torus is a depiction of the frame bundle $\SO(2)\ltimes \setR^2$. Dually to the $\so(2)\ltimes\setR^2$-valued coframe there is a frame $(\xi,e_1,e_2)$.
The construction of the twisted frame bundle is akin to considering the quotient under the left action of $\{\pm1\}\subset \SO(2)$ and renormalising $\xi$ as $\frac12 \xi$. The underlying manifold can be constructed as
\begin{align*}
	 \{\pm1\} \backslash \lp \SO(2) \ltimes \setR^2 \rp
	 	&\to \{\pm1\}\backslash \setR^2\\
	 [g,v] &\mapsto [g\cdot v]
\end{align*}

There is an exceptional $\operatorname{PSO}(2)$ fibre above the origin of $\{\pm1\}\backslash \setR^2$ but the other fibres are of type $\SO(2)$.

\subsubsection{Dynamical generalised frame bundle structures}\label{secno:exdynGFB}

Generalised frame bundles with connections are spaces with the same local structure as frame bundles with connection. Therefore they serve as a useful generalisation when one expects to produce a frame bundle structure, or, more generally, a $G$-structure, from local equations. In General Relativity the causal and gravitational structure of spacetime are encoded in a Lorentzian metric defined on a four-dimensional spacetime. Depending on the variant of the theory, the metric may be supplanted by a coframe field (a so-called \emph{tetrad}) and the metric connection may have supplementary degrees of freedom. The point is that this geometry corresponds to a $G$-structure with connection over the spacetime. Generalised frame bundles with connections provide a new frame for dynamically defined $G$-structures, more precisely a generalisation thereof. In particular, the solder form is the frame bundle equivalent of the tetrad field.

Generalised frame bundles have sufficient structure to define curvature, torsion and matter fields which are all basic fields. This was already put to use by Ne'eman and Regge~\cite{GravGroupManifold} for the specific case of \enquote{group manifolds}. For example, the Einstein tensor is defined as follows. We write $\met^{ab}$ for a $\Glie$-invariant metric on $\setR^4$ and $\rho^b_{i a}$ for the components of the representation (the $i$ index corresponds to $\Glie$)
\[
	\rho : \Glie\to \so(\setR^4,\met)
\]
Recall the curvature $2$-form;
\[
	\Omega = \d \omega + \fwb\omega\omega \in \Omega^2_{bas}(P,\Glie)
\]
The curvature tensor is constructed as
\[
	\Riem = \Omega^i\rho_i \in \Omega^2_{bas}(P,\so(\setR^4,\met))
\]
The associated Ricci tensor is a basic $\setR^{4*}$-valued $1$-form:
\[
	\Ric_d = \Omega^i_{bc} \rho^b_{i d} \alpha^c \in \Omega^1_{bas}(P,\setR^{4*})
\]
and scalar curvature is a basic scalar: 
\[
	\Scal = \Ric_a(\zeta^a) \in \Omega^0_{bas}(P)
\]
The (tetradic) Einstein tensor is defined as
\[
	\Ein_a = \Ric_a - \frac12 \Scal \eta_{ab}\alpha^b \in \Omega^1_{bas}(P,\setR^{4*})
\]
and with a matter field $\psi$ on which depends an (adimensionalised) stress-energy tensor $T \in \Omega^1_{bas}(P,\setR^{4*})$, Einstein's field equation can be formulated as 
\begin{equation}
	\Ein = T \in \Omega^1_{bas}(P,\setR^{4*})
\end{equation}

Following an idea from Toller~\cite{CFTRefFrames} and revisited by Hélein and Vey as well as the author~\cite{LFB,GFB} constructs a Lagrangian field theory on a $10$-dimensional manifold with a $\so(4)\ltimes \setR^{4}$-valued coframe $\varpi$ and a field $\psi$ with value in a spinor representation of $\Spin(4)$ as fields. The theory uses a kind of generalised Lagrange multipliers, and the equations of motion constrain $\varpi$ to be a Cartan $1$-form and impose equations of Einstein-Cartan type and Dirac type on $\varpi$ and $\psi$, which also involve the Lagrange multipliers. It is shown that in the case the generalised $\so(4)$-structure is a standard $\Spin(4)$-structure the field equations can be decoupled from the Lagrange multipliers and the usual Einstein-Cartan-Dirac field equations are recovered on the underlying spacetime (in Riemannian signature).


%
%
%

\section{The structure of $G$-manifolds}\label{secno:HManifolds}
	The generalised frame bundles introduced in Section~\ref{secno:GFB} have the action of a Lie algebra $\Glie$, but a priori they have no group action. On the other hand, frame bundles have the smooth action of a \emph{structural group} $G$, which is furthermore a principal action, so that the frame bundle does indeed form a $G$-principal bundle above a base manifold.

Are all generalised frame bundles actual frame bundles of smooth manifolds?
The example given in Section~\ref{secno:exspinchir} shows that it need not be the case. This brings the question: what is the extent of this generalisation of the frame bundles?
In this section, we review conditions for a Lie algebra action to describe a principal fibre bundle.
In order to tackle this question, we focus on the three following points:
\begin{itemize}
	\item When does a Lie algebra action integrate to a Lie group action?
	\item When is the orbit space of a $G$-manifold a smooth manifold?
	\item When does a $G$-manifold form a $G$-principal bundle above its orbit space?
\end{itemize}
The application to the specific case of generalised frame bundle will be the object of Section~\ref{secno:CartanInteg}

		\subsection{$\Glie$-manifolds}\label{secno:hmanifolds}

In this section we describe the different aspects of the problem of integrating the action of a Lie algebra to a Lie group action.
We do not got into much detail; the reader will find more detail in Palais' thesis~\cite{GlobalLie}, a more recent article by Kamber and Michor~\cite{LieActionInt} or the review in the author's PhD thesis~\cite{PDMPhD}.
We choose a connected Lie group integration $G$ of the Lie algebra $\Glie$.

Let $P$ be a \emph{right} $\Glie$-manifold. To construct the action of an element $g$ of $G$ on a point $p\in P$, a natural idea is to choose a smooth path $\gamma$ in $G$ from $e$ to $g$ and to solve the differential equation 
\begin{equation}\label{eqno:c'gamma}
\begin{cases}
	c'(t) = c(t) \cdot \gamma'(t)\\
	c(0) = p
\end{cases}
\end{equation}

The final endpoint is the candidate for $p\cdot g$.
In order for this procedure to succeed, the differential equation~\eqref{eqno:c'gamma} needs to have a complete solution. A first problem is therefore that of \emph{completeness} of the action of $\Glie$ on $P$. According to \cite{GlobalLie} (Theorem IV.III), in order for all such differential equations to admit complete solutions, it is sufficient that all fundamental vector fields are complete.

Let us assume that all the fundamental vector fields are complete
.
For the above procedure to consistently define an action of $G$, the end of the path $c$ must be independent from the choice of $\gamma$ and only depend on the starting point $p$ of $c$ and the endpoint $g$ of $\gamma$. This condition is called \emph{univalence}~\cite{GlobalLie}, and it is always verified (for a complete action) when $G$ is simply connected. It is sufficient to consider loops $\gamma$ about the identity element. One topological formulation of univalence is that the trivial bundle $P\times G \to P$ equipped with the flat connection corresponding to the right action of $\Glie$ has trivial holonomy.\com{Pas de référence...}

On a compact manifold, all vector fields are complete, so the following proposition holds:
\begin{proposition}\label{propno:ActInteg}
	Let $G$ be a simply-connected Lie group and $P$ a compact $\Glie$-manifold.
	Then the action of $\Glie$ on $P$ integrates to a Lie group action of $G$.
\end{proposition}

	Let us now discuss the case when the action of $\Glie$ on $P$ is incomplete. Then certain parts of the group are unable to act on $P$: only parts of the group may act on $P$, with this part of $G$ depending on the point of $P$. This \enquote{local action} may be conveniently encoded into a Lie groupoid~\cite{PDMPhD}, but we are looking for a group action. It is therefore necessary to \emph{complete} the manifold. There is a generic procedure, described in~\cite{GlobalLie, LieActionInt}, which constructs, starting from $P$, a larger manifold on which $G$ acts. This manifold is called the $G$-completion of $P$ and is not always Hausdorff, since it is constructed as a leaf space. In fact, the $G$-completion of $P$ is not always larger than $P$: when there is a defect of univalence, $P$ is reduced to a suitable quotient so as to satisfy univalence.


	\subsection{Proper actions}\label{secno:ProperActions}

	In this section and the following sections, we consider a \emph{Hausdorff} manifold $P$ on which a Lie group $G$ acts smoothly \emph{on the right}. We want to know under what conditions the quotient $P\to P/G$ defines a $G$-principal bundle. First, we discuss \emph{properness}, which is an appropriate separation condition on $G$-manifolds in order to have a well-behaved quotient. A general reference is Meinrenken's lecture notes~\cite{GroupActions}.
	
	\begin{definition}[Proper action]
		The action of $G$ on $P$ is \emph{proper}
		if the inverse images of compact subsets under the following application are compact:
		\begin{align*}
			P\times G &\to P \times P\\
			(p,g)	&\mapsto (p, p\cdot g)
		\end{align*}
	\end{definition}
	
	\begin{proposition}[\cite{GroupActions}]
		If $G$ is compact then the action on $P$ is proper.
	\end{proposition}
	In fact, properness captures important topological properties of the continuous actions of compact groups.
	
	\begin{proposition}
	\begin{enumerate}
	\item
		If the action of $G$ is proper then the isotropy groups of the elements of $P$ are compact.
	\item \label{propno:OrbitCoset}
		If the action of $G$ is proper then for every $p\in P$, the orbital map
			\[
				g \in G_p \backslash G \mapsto p \cdot g
			\]
			is an embedding onto the orbit $p\cdot g$ of $p$, which gives the orbit the structure of an \emph{embedded manifold}.
	\item 
		If the action of $G$ is proper then the quotient topology on $P/G$ is Haudorff. 
	\end{enumerate}
	\end{proposition}

	\paragraph{Principal $G$-manifolds}
	
	In order for the quotient map $P\to P/G$ to define a $G$-principal bundle, $G$ needs to act freely.
	
	\begin{theorem}[Quotient manifold theorem, \cite{GroupActions}]\label{thmno:QuotientManifold}
	If the action of $G$ on $P$ is free and proper, then the quotient $P/G$ has a structure of a (Hausdorff) quotient manifold and the quotient map 
	\begin{align*}
		P &\to P/G\\
		p &\mapsto [p]
	\end{align*}
	defines a (locally trivial) $G$-principal bundle.
	\end{theorem}
	
	Naturally, when the action is not free but all isotropy groups are identical -- let us call this subgroup $K$ -- then $K$ is necessarily a closed normal subgroup and the action of $G$ factors to a free action of the quotient Lie group $G/K$.
	
	\subsection{Orbit types and decomposition}
	
	In this section, we discuss the structure of the orbit space $P/G$ and the quotient map when the action of $G$ is not assumed to be free.
	
	\subsubsection{Orbit types}
	
	If $x$ and $y$ are two points of $P$ belonging to the same orbit, their isotropy groups are conjugate.
	
	\begin{definition}[Type of an orbit]	
	The \emph{orbit type} (or isotropy type\footnote{Bredon~\cite{CompactTransGroups} makes a difference between \emph{orbit type} and \emph{isotropy type}. We are using here the isotropy type.}) 
	of an orbit $\Orb$ of $G$ is the conjugacy class of its isotropy groups. 
	We write it with brackets $[\Orb]$.
	\end{definition}
	Note that orbits with different types may be diffeomorphic, for example different presentations of the circle group as quotients of the real line $\setR$. On the other hand, assuming properness, orbits with the same type are diffeomorphic (see Proposition~\ref{propno:OrbitCoset}).
			
	Orbit types form a pre-ordered set: we define for coset spaces $O_1$ and $O_2$ the following pre-order\footnote{A transitive and reflexive relation.} relation:
	\[ [O_1] \leqslant [O_2]\]
	if and only if for any $p_1\in O_1, p_2\in O_2$, the isotropy group $G_{p_2}$ is conjugated to a subgroup of the isotropy group $G_{p_1}$. Roughly speaking, an orbit is larger when its isotropy group is smaller. 
	
	This relation may or may not be antisymmetric. The following lemma justifies that it is under some assumptions we state below.
	
	\begin{lemma}[{\cite[Lemma 3.15]{DiffStacks}}]\label{lmno:conjsubgroup}
	Let $G$ be a Lie group and $K$ a closed subgroup which has a finite number of connected components (e.g. is compact).
	
	For any $g$ in $G$, $gKg^{-1}\subset K$ implies that $gKg^{-1}=K$.
	\end{lemma}
	
	\begin{proof}
	Let $g$ be an element of $G$ such that $\Ad_g K\subset K$.
	
	At the level of the Lie algebra, since $\Ad_g$ is a linear automorphism, dimension considerations imply that $\Ad_g \Klie \subset \Klie \implies \Ad_g \Klie = \Klie$. The subgroups $\Ad_g K$ and $K$ thus have the same neutral component which we call $K_0$ and which is preserved by $\Ad_g$. The action of $g$ thus factors to the quotient group $K/K_0$, which is finite by assumption.
	
	As an injection of the finite set $K/K_0$ into the subset 
	\[\Ad_h (K/K_0) = (\Ad_h K)/K_0 \subset K/K_0\]
	the map $\Ad_g$ necessarily defines a bijection of $K/K_0$. One concludes that $\Ad_g K$ has a point in each connected component of $K$, thus is equal to $K$.
	
	\end{proof}
	
	As a consequence of Lemma~\ref{lmno:conjsubgroup}, when $G$ is compact, so that all its closed subgroups are compact, the pre-order relation on orbit types is antisymmetric hence an order relation. More generally, when the action of $G$ on $P$ is proper, so that all isotropy groups are compact, the pre-order relation \emph{on the orbit types of $P$} is antisymmetric.

	\subsubsection{Decomposition of a proper $G$-manifold}
	
	When the $G$-manifold $P$ has different orbit types, one cannot hope for a fibre bundle $P\to P/G$ since all orbits are not diffeomorphic as $G$-manifolds. It is therefore required to deal separately with each orbit type.
	
	If $[\Orb]$ is an $G$-orbit type then we write $P_{[\Orb]}$ for the reunion of the orbits of type $[\Orb]$ in $P$. We call $P_{[\Orb]}$ the \emph{part of type $[\Orb]$} of $P$. Not only orbits are embedded submanifolds (Proposition~\ref{propno:OrbitCoset}) but the part of $P$ of any given type is also an embedded submanifold:
	
	\begin{theorem}[\cite{GroupActions}]
		If $P$ is a proper $G$ space, then for any orbit type $[\Orb]$, the corresponding part $P_{[\Orb]}$ is an embedded submanifold of $P$.
	\end{theorem}
	The $G$-manifold $P$ thus decomposes into a reunion of parts which are embedded submanifolds.
	
	\subsubsection{Decomposition of the orbit space}\label{secno:OrbitSpace}
	
	Similarly to the decomposition of $P$, there is a decomposition of $P/G$ according to orbit types. 
	Let us consider $P_{[\Orb]}/G \subset P/G$ the part of the orbit space of type $[\Orb]$.
	
	\begin{theorem}[\cite{GroupActions}]
		Assume the action of $G$ on $P$ is proper. Let $\Orb$ be an orbit type. Then
		\begin{itemize}
		\item $P_{[\Orb]}/G$ has a quotient manifold structure,
		\item The quotient map $P_{[\Orb]} \to P_{[\Orb]}/G$ defines a fibre bundle with each fibre a right $G$-manifold isomorphic to $\Orb$.
		\end{itemize}
	\end{theorem}
	
	In other words, the quotient map $P\to P/G$ does not in general define a fibre bundle but can be decomposed into a collection of locally trivial fibrations.
	The connected components of the subsets 
	$P_{[\Orb]}/G$
	form a so-called \enquote{stratification} of $P$ (see for example Meinrenken's lecture notes~\cite{GroupActions}) but this structure will not be relevant for our purposes.
	
	Globally, the structure of the orbit space of a proper $G$-manifold is captured by the notion of \emph{orbispace} (discussed in detail in~\cite{DiffStacks}), which are spaces locally isomorphic to the quotient of a Euclidean space by a compact group. In the case the action of $G$ is infinitesimally free, thus the isotropy groups are finite, $P/G$ has the structure of an \emph{orbifold}.
	
	\subsubsection{The principal orbit type}
	
	The quotient maps $P_{[\Orb]} \to P_{[\Orb]}/G$ define fibre bundles but in the situation of generalised frame bundles, one could hope for the entirety of $P$ to form a fibre bundle above its orbit space. Although it is not the case, it is true for a large part of $P$, as we explain here.
	
	\begin{theorem}[Principal orbit type, \cite{GroupActions}]\label{thmno:PrincipalOrbits}\label{thmno:PrinOrb}
		If $P$ is a proper $G$-manifold, then it admits a maximal orbit type, which is unique.
		The corresponding part of $P$ is a dense open subset of $P$ and its orbit space is connected, open and dense inside $P/G$.
		
		The maximal orbit type is called the \emph{principal orbit type}, the orbits are called \emph{principal orbits} and the isotropy groups \emph{principal isotropy groups}.
	\end{theorem}
	
	The quotient $P/G$ is a manifold outside of a singular locus, bounds on its dimension can be found in Bredon's monograph~\cite{CompactTransGroups} or in Meinrenken's lecture notes~\cite{GroupActions}.
	
	The principal orbit can be characterised as follows:
	\begin{theorem}[\cite{CompactTransGroups}]\label{thmno:PrincipalOrbitNeighbourhood}
		Assume $G$ acts properly on $P$.
		Then an orbit $\Orb \subset P$ is principal if an only if it admits a $G$-invariant neighbourhood $\U \supset \Orb$ that is isomorphic as a $G$-manifold to $\Orb \times \setR^k$ for some $k$, with $G$ acting trivially on $\setR^k$.
	\end{theorem}

%
%
	

\section{Integration of a Cartan $1$-form}\label{secno:CartanInteg}
In this section we finally deal with the question of building a principal bundle structure starting from a manifold $P$ equipped with a Cartan $1$-form $\varpi$ with value in a Lie algebra $\Glie\ltimes \setR^n$. Recall the following Maurer-Cartan-like equation the Cartan form $\varpi$ is required to satisfy: 
\begin{equation}\label{eqno:introCartform}
	\d\varpi + \fwb\varpi\varpi = \frac12\Omega_{bc}\alpha^a\wedge\alpha^b
\end{equation}
with $\alpha=\varpi_\setR^n$ the $\setR^n$ component of $\varpi$ and $\Omega_{bc}$ unconstrained coefficients each with value in $\Glie\ltimes \setR^n$.

\subsection{Integration of a Cartan 1-form}

\subsubsection{Cartan $1$-forms}

We showed in Section~\ref{secno:GFB} that $P$ is equipped with a free action of the Lie algebra $\Glie$ and that $\varpi$ is equivariant under this action.

If the vector fields are complete the Lie algebra action is readily integrated into a Lie group action of the simply connected Lie group integration $\tilde{G}$ of $\Glie$. If not, one needs to \emph{complete} the manifold $P$ as explained in Section~\ref{secno:hmanifolds}. The Cartan form can be shown to uniquely extend to a Cartan form on the completion~\cite{PDMPhD}. Since the isotropy algebras are trivial, the isotropy groups are necessarily discrete subgroups.

Let us introduce the following intermediate notion between Cartan $1$-form and actual Cartan connection $1$-forms on principal bundles:
\begin{definition}[Generalised Cartan connection]
	Let $\Glie$ be a Lie algebra acting on $\setR^n$.
	Let $P$ be a $\Glie$-manifold. A \emph{generalised $\Glie\ltimes\setR^n$-valued Cartan connection} on $P$ is a nondegenerate $\Glie\ltimes \setR^n$-valued $1$-form $\varpi$ which is normalised for the action of $\Glie$ and satisfies the equivariance equation: 
\begin{empheq}[left={\forall \h\in\Glie, \ \empheqlbrace}]{align}
	&\varpi \lp \bh \rp = \h\\
	&\Lie_{\bh} \varpi + \ad_\h \varpi = 0\label{eqno:varpiequiv}
\end{empheq}
\end{definition}
\begin{remark}
Note that our definition differs from the one in~\cite{DiffGeoCartan}, as theirs does not impose that the form is nondegenerate, but is close to their definition of \emph{principal Cartan connection} (which requires the structure of a principal bundle). 
\end{remark}
Since $G$-manifolds are naturally $\Glie$-manifolds, generalised Cartan connections also make sense on $G$-manifolds.

Assuming that the problem of integrating the Lie algebra action to a group action was solved, a Cartan $1$-form becomes a generalised Cartan connection.
Let us now see to what extent generalised Cartan connections are related to Cartan connections on principal bundles.

\subsubsection{Principal orbits and bundle fibration}

In this section, $P$ is a connected manifold with a Cartan form $\varpi$, which defines a (free) infinitesimal action of $\Glie$. We now assume that the infinitesimal action integrates to a \emph{group} action of $G$. Namely we require univalence and completeness, as explained in Section~\ref{secno:hmanifolds}. We want to construct from $P$ a principal bundle with a Cartan connection. For this the action needs to suitable isotropy groups, as described in Theorem~\ref{thmno:QuotientManifold} and the subsequent comment.

We further require the action to be \emph{proper}.
As a consequence, the isotropy groups are compact, and the action of $\Glie$ being free, they are discrete, hence finite:
\begin{lemma}\label{thmno:DiscIso}
	Let $G$ be a Lie group and $P$ a proper $G$-manifold.
	If the infinitesimal action of $\Glie$ is free then the isotropy groups of $P$ are finite.
\end{lemma}

We restrict our attention to $\Pprin$, the principal part of $P$, which is a dense open subspace (see Theorem~\ref{thmno:PrincipalOrbits}).

We need an effective group action on $\Pprin$ in order to construct a principal fibre bundle. This requires all isotropy groups to be identical, so that there is one uniquely defined quotient group of $G$ acting on $\Pprin$. In particular, the isotropy groups have to be normal subgroups of $G$, although this is no sufficient condition.

\begin{theorem}
Let $P$ be a connected $G$-manifold equipped with a $\Glie\ltimes\setR^n$-valued generalised Cartan connection.
Assume that the action of $G$ on $P$ is proper and has a single orbit type. The isotropy groups are then finite.

Assume furthermore that the isotropy groups are normal subgroups of $G$. They are then all identical. Let us call $K$ the isotropy group. 
The manifold $P$ has a free and effective action of $K \backslash G$ and forms a principal $K \backslash G$-bundle over the quotient $P/(K \backslash G)$ which has a quotient differentiable manifold structure.
The principal bundle is equipped with a solder form and a $(K \backslash G)$-principal connection.
\end{theorem}

\begin{proof}
The action of $G$ naturally factors to a free action of $K \backslash G$. Theorem~\ref{thmno:QuotientManifold} applies and the orbit space $P/(K \backslash G)\simeq P/G$ is a manifold over which $P$ forms a principal $K \backslash G$-bundle. Since $K$ is discrete, the Lie algebras of $K \backslash G$ is identified with $\Glie$. The vertical distribution integrates to the fibres of the principal bundle fibration. The $\setR^n$ component of the generalised Cartan connection on $P$ defines a solder form and the $\Glie$ component a $K\backslash G$-principal connection on $P$.
\end{proof}

The Cartan connection imposes constraints on the isotropy subgroups. Because we restrict our attention to the principal isotypic component, isotropy groups act trivially on the normal tangent bundle to the orbits (Theorem~\ref{thmno:PrinOrb}). The tangent spaces to the orbits constitute the (integrable) distribution $(\ker \alpha)$, so $\alpha^a$ forms a coframe of the normal bundle to the orbits. Furthermore $\alpha$ is by hypothesis equivariant under $G$ so that it is equivariant, hence invariant, under the isotropy groups of the principal orbits. One concludes that the principal isotropy groups act trivially on $\setR^n$:
\begin{lemma}\label{thmno:IsoKernel}
Let $P$ be a proper $G$-manifold with a generalised $(\Glie\ltimes\setR^n)$-valued Cartan connection $\varpi$.\\
The principal isotropy groups of $P$ are subgroups of the kernel of the action
\[
	G \xrightarrow{\Ad} \End_\setR (\setR^n)
\]
\end{lemma}
In particular, when the action of $G$ on $\setR^n$ is faithful, the principal isotropy groups have to be trivial.

\begin{proof}
	Let $\Orb \subset P$ be a principal orbit.
	According to Theorem~\ref{thmno:PrincipalOrbitNeighbourhood}, $\Orb$ admits a neighbourhood in $P$ which is $G$-equivariantly diffeomorphic to $\Orb \times \setR^k$ for some $k\in \setN$ and we can assume $P=\Orb\times\setR^k$ without loss of generality. The vertical distribution is $T\Orb \times 0_{T\setR^k}$ and the horizontal distribution is $0_{T\Orb} \times T\setR^k$.
	Since $\alpha$ induces a trivialisation $HP \isom \setR^n \times P$, necessarily $k=n$.

	Let $p\in \Orb$ and $g\in \Stab_G (p)$. Then $g$ preserves $HP$ and $VP$, moreover since $P= {\Orb\times \underbrace{\setR^n}_{\text{trivial action}} }$, the action of $g$ on horizontal vectors in $H_p P$ is trivial. 
	But $\alpha$ is $G$-equivariant thus
	\[
	\forall u\in H_p P, \quad
		g\cdot \alpha (u)
			= \alpha(u\cdot g^{-1})
			= \alpha(u) 
	\]
	Finally, $\alpha(T_p P) = \setR^n$ so that $g$ needs to act trivially on the linear representation $\setR^n$
\end{proof}

\begin{example}
Let $P$ be a $\SO(n)$-manifold.

Let $\varpi$ be a generalised $\so(n)\ltimes \setR^n$-valued Cartan connection on $P$. Then the action of $\SO(n)$ is both proper and free over the principal isotypic component $\Pprin$ (due to Lemma~\ref{thmno:IsoKernel}). It thus defines a principal bundle $\Pprin\to \Pprin/\SO(n)$ over the orbit space which is a smooth Hausdorff manifold, and $\varpi$ defines an $\SO(n)$-principal connection.
\end{example}

The examples of locally Klein geometries introduced in Section~\ref{secno:examples} are examples of this construction. In particular, Example~\ref{secno:exspinchir} has an action of $\Spin(4)$ which is free except on one single exceptional orbit which induces a localised singularity on the orbit space.

\subsection{$\Glie\ltimes\setR^n$-valued Cartan $1$-form on a compact manifold with compact $G$}

Under compactness assumptions, many hypotheses of the construction become automatically satisfied. 
Let $P$ be a compact manifold. 
Let $G$ be a compact and simply-connected Lie group with a linear action on $\setR^n$. 
Assume that $P$ is equipped with a $\Glie\ltimes\setR^n$-valued Cartan $1$-form $\varpi$.

The Cartan $1$-form defines a free action of $\Glie$ on $P$ according to Theorem~\ref{thmno:CartformAct}. Since the manifold $P$ is compact, the action of $\Glie$ is complete. Hence the infinitesimal action integrates to a Lie group action of the simply-connected integration $G$ of $\Glie$, as stated in Section~\ref{secno:hmanifolds}.

Since $\tilde{G}$ is compact, the action is necessarily \emph{proper}.
In particular, it has a principal orbit type. 
The principal part of $P$ is a dense open subset (Theorem~\ref{thmno:PrinOrb}), its orbit space $\Pprin/G$ is a smooth manifold according to Theorem~\ref{thmno:QuotientManifold} and the quotient map $\Pprin \to \Pprin/G$ defines a fibre bundle.
All the isotropy groups are necessarily finite (Lemma~\ref{thmno:DiscIso}). This is summed up in the following Lemma:
\begin{lemma}
Let $G$ be a simply connected compact Lie group with Lie algebra $\Glie$ and $P$ a compact $\Glie$-manifold. 
Then the action of $\Glie$ on $P$ integrates into a group action of $G$ and there is a finite subgroup $K\subset G$ and a dense open subset $\U\subset P$ stable under $G$ such that
$\U/G$ is a smooth manifold and the map
\[ \U\to\U/G \]
defines a fibre bundle with typical fibre $K \backslash G$ which is an homogeneous $G$-manifold covered by $G$.
\end{lemma}

When the action of $G$ on $\setR^n$ is faithful, Lemma~\ref{thmno:IsoKernel} implies that the principal isotropy groups are necessarily trivial.

\begin{theorem}
Let $P$ be a compact manifold. Let $G$ be a simply connected compact Lie group with a Lie algebra $\Glie$ and a faithful action on $\setR^n$. Let $\varpi=\omega\oplus\alpha\in\Omega^1(P,\Glie\ltimes\setR^n)$ be a generalised Cartan connection on $P$. 

Then there exists a dense open subset $\U\subset P$ stable under $G$ such that $\U/G$ is a smooth manifold, 
\[ \U\to\U/G \]
is an $G$-principal bundle, $\alpha$ is a solder form and $\omega$ is a $G$-principal connection on the fibre bundle.
\end{theorem}

\begin{example}
Let $P$ be a compact manifold equipped with a $\spin(n)\ltimes \Sigma^n$-valued Cartan $1$-form $\varpi$, with $\Sigma^n$ a faithful spinorial representation of $\spin(n)$. 

The Cartan $1$-form induces an action of $\spin(n)$ on $P$, which is necessarily complete and integrates to a group action of $\Spin(n)$, under which $\varpi$ is equivariant. 
The action of $\Spin(n)$ is proper and since $\Sigma^n$ is a faithful representation of $\Spin(n)$ the principal isotypic component $\Pprin$ is free under the group action.

One concludes that $\Pprin$, the principal part of $P$, defines a principal bundle $\Pprin\to \Pprin/\Spin(n)$ over its orbit space, which is a smooth Hausdorff manifold, and $\varpi$ defines a solder form and a $\Spin(4)$-principal connection.
\end{example}

Now, in the general case, the total orbit space is an orbifold (as mentionned in Section~\ref{secno:OrbitSpace}). Even in this case, the structure we obtain can be understood as a principal connection on the frame bundle. The theory of frame bundles and connections on orbifolds is exposed in~\cite{GStructOrbi}.


\printbibliography

\end{document}